\documentclass[11pt,a4paper]{article}

\usepackage{epsf,epsfig,amsfonts,amsgen,amsmath,amstext,amsbsy,amsopn,amsthm,cases,listings,color
}
\usepackage{ebezier,eepic}
\usepackage{color}
\usepackage{multirow}
\usepackage{epstopdf}
\usepackage{graphicx}   
\usepackage{pgf,tikz}
\usepackage{mathrsfs}
\usepackage[marginal]{footmisc}
\usepackage{enumitem}
\usepackage[titletoc]{appendix}
\usepackage{booktabs}
\usepackage{url}
\usepackage{mathtools}

\usepackage{pgfplots}
\usepackage{authblk}
\usepackage{amssymb}

\usepackage{wasysym}

\usepackage{empheq}

\usepackage{dsfont}

\usepackage{tikz}
\usepackage{longtable}

\usepackage{subfigure}
\pgfplotsset{compat=1.18}
\usepackage{mathrsfs}
\usepackage{wasysym} 
\usetikzlibrary{arrows}
\usepackage{aligned-overset}
\usepackage{bm}
\usepackage{bbm}
\usepackage[backref=page]{hyperref}

\usepackage{float}

\allowdisplaybreaks[1]

\definecolor{uuuuuu}{rgb}{0.27,0.27,0.27}
\definecolor{sqsqsq}{rgb}{0.1255,0.1255,0.1255}

\newtheorem{definition}{Definition} [section]
\newtheorem{theorem}[definition]{Theorem}
\newtheorem{lemma}[definition]{Lemma}
\newtheorem{proposition}[definition]{Proposition}

\newtheorem{claim}[definition]{Claim}

\newtheorem{fact}[definition]{Fact}
\newenvironment{pf}{\noindent{\bf Proof}}

\newcommand{\multiset}[1]{\{\hspace{-0.25em}\{\hspace{0.1em}#1\hspace{0.1em}\}\hspace{-0.25em}\}}

\begin{document}
\title{\bf\Large Spectral generalized Tur\'{a}n problems}
\date{\today}
\author[ ]{Xizhi Liu\thanks{Research supported by the Excellent Young Talents Program (Overseas) of the National Natural Science Foundation of China. Email: \texttt{liuxizhi@ustc.edu.cn}}}
\affil[ ]{School of Mathematical Sciences, 
            USTC,
            Hefei, 230026, China}
\maketitle
\begin{abstract}
Combining two well-studied variants of the classical Tur\'{a}n problem, the generalized Tur\'{a}n problem and the spectral Tur\'{a}n problem, we introduce the spectral generalized Tur\'{a}n problem and establish a general theorem that extends the result of Keevash--Lenz--Mubayi~\cite{KLM14} on the spectral Tur\'{a}n problem in this broader setting.   
As a quick application, we obtain the spectral Erd\H{o}s Pentagon Theorem.  
We also introduce the notion of entropic density for generalized Tur\'{a}n problems, and show that it coincides with the generalized spectral radius, extending a recent result of Chao--Hans on entropic Tur\'{a}n density.
\end{abstract}
\section{Introduction}\label{SEC:Introduction}
For an integer $r\ge 2$, an \textbf{$r$-uniform hypergraph} (henceforth $r$-graph) $\mathcal{H}$ is a collection of $r$-subsets of some finite set $V$.
We identify a hypergraph $\mathcal{H}$ with its edge and use $|\mathcal{H}|$ to denote the number of edges in it. 
The vertex set of $\mathcal{H}$ is denoted by $V(\mathcal{H})$, and its size is denoted by $v(\mathcal{H})$.

Given a family $\mathcal{F}$ of $r$-graphs, an $r$-graph $\mathcal{H}$ is $\mathcal{F}$-free if it does not contain any member of $\mathcal{F}$ as a subgraph. 
Following the seminal work of Tur\'{a}n~\cite{TU41}, a central problem in extremal combinatorics is to determine the maximum number of edges in an $\mathcal{F}$-free $r$-graph on $n$ vertices, known as the Tur\'{a}n number of $\mathcal{F}$:
\begin{align*}
    \mathrm{ex}(n,\mathcal{F})
    & \coloneqq \max\left\{ |\mathcal{H}| \colon \text{$v(\mathcal{H}) = n$ and $\mathcal{H}$ is $\mathcal{F}$-free}\right\}.
\end{align*}
An extensively studied variant of the Tur\'{a}n problem is the generalized Tur\'{a}n problem, initiated by Erd\H{o}s~\cite{Erdos62} and studied systematically by Alon--Shikhelman~\cite{AS16}. 
This line of research investigates the maximum number of copies of a fixed $r$-graph $Q$, rather than edges, in an $\mathcal{F}$-free $r$-graph on $n$ vertices.
More specifically, let $\mathrm{N}(Q,\mathcal{H})$ denote the number of copies\footnote{The precise definition will be given in the next subsection.} of $Q$ in $\mathcal{H}$. The generalized Tur\'{a}n problem concerns the following function:
\begin{align*}
    \mathrm{ex}(n, Q, \mathcal{F})
    & \coloneqq \max\left\{ \mathrm{N}(Q,\mathcal{H}) \colon \text{$v(\mathcal{H}) = n$ and $\mathcal{H}$ is $\mathcal{F}$-free}\right\}.
\end{align*}
Another extensively studied variant is the spectral Tur\'{a}n problem, proposed by Nikiforov (see~\cite{BS86,Nik07,Nik10}), which concerns the maximum spectral radius of an $n$-vertex $\mathcal{F}$-free $r$-graph.
Formally, given an $r$-graph $\mathcal{H}$ on the vertex set $[n]$, the Lagrangian polynomial\footnote{Also referred to as the polynomial form of $\mathcal{H}$ in Nikiforov’s survey~\cite{Nik14}.} of $\mathcal{H}$ is defined by
\begin{align}\label{equ:Lagrangian-polynomial}
    P_{\mathcal{H}}(X_1, \ldots, X_n)
    \coloneqq r! \sum_{e \in \mathcal{H}} \, \prod_{i \in e} X_i.
\end{align}
For a real number $\alpha > 0$, define the domain
\begin{align*}
    \Delta^{n-1}_{\alpha}
    \coloneqq 
    \left\{(x_1, \ldots, x_{n}) \in \mathbb{R}^n \colon |x_{1}|^{\alpha} + \cdots + |x_{n}|^{\alpha} = 1 \right\}.
\end{align*}
The \textbf{$\alpha$-spectral radius} of $\mathcal{H}$ is then given by  
\begin{align*}
    \lambda_{\alpha}(\mathcal{H})
    \coloneqq \max\left\{P_{\mathcal{H}}(x_1, \ldots, x_{n}) \colon (x_1, \ldots, x_{n}) \in \Delta^{n-1}_{\alpha} \right\}. 
\end{align*}
The spectral Tur\'{a}n problem asks for the value of the following function:
\begin{align*}
    \mathrm{spex}_{\alpha}(n,\mathcal{F})
    & \coloneqq \max\left\{\lambda_{\alpha}(\mathcal{H}) \colon \text{$v(\mathcal{H}) = n$ and $\mathcal{H}$ is $\mathcal{F}$-free}\right\}. 
\end{align*}

In this work, we initiate the study of the spectral generalized Tur\'{a}n problem and extend the general result of Keevash--Lenz--Mubayi~\cite{KLM14} on spectral Tur\'{a}n problem to the spectral generalized Tur\'{a}n problem setting. 
To provide a gentle introduction, we begin with a classical example from the generalized Tur\'{a}n problem.

\subsection{Example: the spectral Erd\H{o}s Pentagon Problem}
Let $C_5$ denote the graph on the vertex set $\{1,2,3,4,5\}$ with edge set 
\begin{align*}
    \big\{\{1,2\}, \{2,3\}, \{3,4\}, \{4,5\}, \{5,1\}\big\}.
\end{align*}
A graph $G$ is said to be \textbf{$C_5$-colorable} if there exists a homomorphism from $G$ to $C_5$, that is, a map $\phi \colon V(G) \to V(C_5)$ such that $\phi(e) \in C_5$ for every $e \in G$. 
Let $\mathfrak{C}_{5}$ denote the collection of all $C_{5}$-colorable graphs. 

In 1984, Erd\H{o}s~\cite{Erd84} conjectured that for every $n \in \mathbb{N}$, the triangle-free graph on $n$ vertices that contains the maximum number of copies of $C_5$ is $C_5$-colorable. 
Building on Razborov’s powerful flag algebra machinery~\cite{Raz07}, his conjecture was resolved for large $n$ independently by Grzesik~\cite{Gre12} and Hatami--Hladk\'{y}--Kr\'{a}\v{l}--Norine--Razborov~\cite{HHKNR13}, and later for all $n$ by Lidick\'{y}--Pfender~\cite{LP18}.  

\begin{theorem}[\cite{Gre12,HHKNR13,LP18}]\label{THM:Pentagon-exact}
    For every $n \in \mathbb{N}$,  
    \begin{align*}
        \mathrm{ex}(n,C_5,K_{3})
        = \max\left\{\mathrm{N}(C_5, G) \colon \text{$G \in \mathfrak{C}_5$ and $v(G) = n$}\right\}.
    \end{align*}
\end{theorem}

We now introduce the spectral version of Theorem~\ref{THM:Pentagon-exact}. 

Given a graph $G$, and for convenience assuming that its vertex set is $[n]$, we define the \textbf{Lagrangian $C_5$-polynomial} of $G$ as
\begin{align*}
    P_{C_5, G}(X_1, \ldots, X_{n})
    \coloneqq \sum_{\phi \in \mathrm{Inj}(C_5, G)} \, \prod_{i\in \phi(V(C_5))} X_{i}, 
\end{align*}
where $\mathrm{Inj}(C_5, G)$ is the collection of all injective homomorphisms from $C_5$ to $G$.

For every $\alpha \ge 1$, the \textbf{$(\alpha, C_5)$-spectral radius} of $G$ is defined as
\begin{align*}
    \lambda_{\alpha, C_5}(G)
    \coloneqq \max\left\{P_{C_5, G}(x_1, \ldots, x_{n}) \colon (x_1, \ldots, x_{n}) \in \Delta^{n-1}_{\alpha} \right\}. 
\end{align*}
The spectral analogue of the Erd\H{o}s Pentagon Theorem is as follows. 
\begin{theorem}\label{THM:C5-K3-spectral}
    Let $\alpha > 1$ be a real number. 
    There exists a constant $N_{\ref{THM:C5-K3-spectral}}$ such that for every $n \ge N_{\ref{THM:C5-K3-spectral}}$, every $K_{3}$-free graph $G$ on $n$ vertices satisfies 
    \begin{align*}
        \lambda_{\alpha}(G)
        \le \max\left\{ \lambda_{\alpha}(H) \colon \text{$H \in \mathfrak{C}_5$ and $v(H) = n$}\right\},
    \end{align*}
    with equality only if $G \in \mathfrak{C}_5$. 
\end{theorem}

\subsection{General definitions and main result}
In this subsection, we introduce the general definitions for the spectral generalized Tur\'{a}n problems and state the main result.

Given two $r$-graphs $Q$ and $\mathcal{H}$, a map $\phi \colon V(Q) \to V(\mathcal{H})$ is called a \textbf{homomorphism} from $Q$ to $\mathcal{H}$ if $\phi(e) \in \mathcal{H}$ for every edge $e \in Q$.
Let $\mathrm{Inj}(Q,\mathcal{H})$ be the set of all injective homomorphisms\footnote{For clarity, we assume that every $r$-graph is vertex-labeled, with distinct vertices assigned distinct labels.} from $Q$ to $\mathcal{H}$, and write $\mathrm{inj}(Q,\mathcal{H}) \coloneqq |\mathrm{Inj}(Q,\mathcal{H})|$. 
The number of copies of $Q$ in $\mathcal{H}$ is given by
\begin{align}\label{equ:def-number-Q}
    \mathrm{N}(Q, \mathcal{H})
    \coloneqq \mathrm{inj}(Q, \mathcal{H})/{|\mathrm{Aut}(Q)|}, 
\end{align}
where $\mathrm{Aut}(Q) \coloneqq \mathrm{Inj}(Q, Q)$ denotes the automorphism group of $Q$.

Following the definition in~\cite{CL24}, we define the \textbf{Lagrangian $Q$-polynomial} of a hypergraph $\mathcal{H}$ as
\begin{align*}
    P_{Q,\mathcal{H}}(X_1, \ldots, X_{n})
    \coloneqq \sum_{\phi \in \mathrm{Inj}(Q,\mathcal{H})} \prod_{i\in \phi(V(Q))}X_i. 
\end{align*}
Note that in the special case where $Q = K_r^r$, this reduces to the classical Lagrangian polynomial given in~\eqref{equ:Lagrangian-polynomial}, which was first introduced by Motzkin--Straus~\cite{MS65} for graphs and later extended to hypergraphs by Frankl--R\"{o}dl~\cite{FR84}. 

Define the \textbf{$(\alpha, Q)$-spectral radius} of a hypergraph $\mathcal{H}$ by
\begin{align*}
    \lambda_{\alpha, Q}(\mathcal{H})
    \coloneqq \max\left\{P_{Q,\mathcal{H}}(x_1, \ldots, x_{n}) \colon (x_1, \ldots, x_{n}) \in \Delta^{n-1}_{\alpha} \right\}. 
\end{align*}
The \textbf{$(\alpha, Q)$-spectral Tur\'{a}n number} of $\mathcal{F}$ is then defined as 
\begin{align*}
    \mathrm{spex}_{\alpha}(n,Q,\mathcal{F})
    & \coloneqq \max\left\{\lambda_{\alpha, Q}(\mathcal{H}) \colon \text{$v(\mathcal{H}) = n$ and $\mathcal{H}$ is $\mathcal{F}$-free}\right\}.  
\end{align*}
%


Suppose that the number of vertices in $Q$ is $q$.
For our purposes, it will be convenient to work with the following definition instead of $\mathrm{ex}(n, Q, \mathcal{F})$: 
\begin{align*}
    \mathrm{inj}(n,Q,\mathcal{F})
    \coloneqq \max\left\{\mathrm{inj}(Q,\mathcal{H}) \colon  \text{$v(\mathcal{H}) = n$ and $\mathcal{H}$ is $\mathcal{F}$-free} \right\}. 
\end{align*}
Define the limit  
\begin{align*}
    \hat{\pi}(Q,\mathcal{F})
    \coloneqq \lim_{n \to \infty} {\mathrm{inj}(n,Q,\mathcal{F})}/{n^{q}}. 
\end{align*}
Note from~\eqref{equ:def-number-Q} that $\mathrm{ex}(n,Q,\mathcal{F}) = {\mathrm{inj}(n,Q,\mathcal{F})}/{|\mathrm{Aut}(Q)|}$. Thus 
\begin{align*}
    \hat{\pi}(Q,\mathcal{F})
    = q! \cdot |\mathrm{Aut}(Q)| \cdot \pi(Q,\mathcal{F}), 
    \quad\text{where}\quad 
    \pi(Q,\mathcal{F}) 
    \coloneqq \lim_{n \to \infty} {\mathrm{ex}(n,Q,\mathcal{F})}/{\binom{n}{q}}. 
\end{align*}
By substituting the vector $(n^{-1/\alpha}, \ldots, n^{-1/\alpha}) \in \Delta_{\alpha}^{n-1}$ into the polynomial $P_{Q,\mathcal{H}}$, where $\mathcal{H}$ is an $r$-graph on $n$-vertices, we obtain the following simple inequalities.
\begin{fact}\label{FACT:lower-bound-spex}
    Let $\alpha \ge 1$ be a real number and $r \ge 2$  be an integer. 
    Let $Q$ be an $r$-graph on $q$ vertices. 
    Then, for every $n \in \mathbb{N}$ and every $n$-vertex $r$-graph $\mathcal{H}$, 
    \begin{align*}
        \lambda_{\alpha, Q}(\mathcal{H}) \ge |\mathcal{H}| n^{-\frac{q}{\alpha}}. 
    \end{align*}
    In particular, for every family $\mathcal{F}$ of $r$-graphs, 
    \begin{align*}
        \mathrm{spex}_{\alpha}(n,Q,\mathcal{F})
        \ge {\mathrm{inj}(n,Q,\mathcal{F})} \cdot n^{-\frac{q}{\alpha}}.
    \end{align*}
\end{fact}

Conversely, Fact~\ref{FACT:lower-bound-spex} provides an upper bound on the generalized Tur\'{a}n number: 
\begin{align*}
    \mathrm{ex}(n,Q, \mathcal{F})
    \le \frac{\mathrm{spex}_{\alpha}(n,Q,\mathcal{F})}{|\mathrm{Aut}(Q)|} n^{\frac{q}{\alpha}}. 
\end{align*}
We now introduce definitions that extend the corresponding notions from the work of Keevash–Lenz–Mubayi~\cite{KLM14} to the setting of the generalized Tur\'{a}n problem.

Fix an $r$-graph $Q$ and let $q \coloneqq v(Q)$ be the number of vertices in $Q$. 
Let $\mathcal{F}$ be a family of $r$-graphs and $\mathfrak{H}$ be a hereditary family of $\mathcal{F}$-free $r$-graphs. 
Here, \textbf{hereditary} means that for every $\mathcal{H} \in \mathfrak{H}$, all subgraphs of $\mathcal{H}$ also belong to $\mathfrak{H}$. 
For every $n \in \mathbb{N}$, let 
\begin{align*}
    \lambda_{\alpha, Q}(n, \mathfrak{H})
    \coloneqq \max\left\{\lambda_{\alpha, Q}(\mathcal{H}) \colon \text{$\mathcal{H} \in \mathfrak{H}$ and $v(\mathcal{H}) = n$}\right\}.
\end{align*}
For an $r$-graph $\mathcal{H}$ and a vertex $v\in V(\mathcal{H})$, the \textbf{$Q$-degree} of $v$ in $\mathcal{H}$ is defined as 
\begin{align*}
    d_{Q,\mathcal{H}}(v)
    \coloneqq \left\{\phi \in \mathrm{Inj}(Q,\mathcal{H}) \colon v\in \phi(V(Q))\right\}. 
\end{align*}
We use $d_{Q}(\mathcal{H})$ and $\delta_{Q}(\mathcal{H})$ to denote the \textbf{average} and \textbf{minimum} $Q$-degree of $\mathcal{H}$, respectively.
Since every member in $\mathrm{Inj}(Q, \mathcal{H})$ has exactly $q$ vertices, we have  
\begin{align*}
    d_{Q}(\mathcal{H})
    = \frac{1}{n} \sum_{v\in V(\mathcal{H})} d_{Q,\mathcal{H}}(v)
    = \frac{q \cdot \mathrm{inj}(Q, \mathcal{H})}{n}. 
\end{align*}

\begin{definition}\label{DEF:generalized-spectral}
    Let $\alpha \ge 1$ and $\delta \ge 0$ be real numbers. Let $N \ge 1$ be an integer. 
    \begin{enumerate}[label=(\roman*)]
        \item\label{DEF:Q-smoothness} We say that $\mathcal{F}$ is $(\delta, N, Q)$-smooth if for every $n \ge N$ the following inequality holds: 
        \begin{align}\label{equ:Turan-smoothness}
            \left|\mathrm{inj}(n, Q, \mathcal{F}) - \mathrm{inj}(n-1, Q, \mathcal{F}) - q \hat{\pi}(Q, \mathcal{F}) \cdot n^{q-1}\right|
            \le \delta n^{q-1}. 
        \end{align}
        \item\label{DEF:Q-degree-stable} We say that $\mathcal{F}$ is $(\delta, N, Q)$-degree stable with respect to $\mathfrak{H}$ if every $\mathcal{F}$-free $r$-graph $\mathcal{H}$ on $n \ge N$ vertices with $\delta_{Q}(\mathcal{H}) \ge (1-\delta) q \hat{\pi}(Q, \mathcal{F}) n^{q-1}$ is contained in $\mathfrak{H}$.
        \item\label{DEF:spectral-balanced} We say that $\mathfrak{H}$ is $(\alpha, \delta, N, Q, \mathcal{F})$-balanced in spectral if for every $n \ge N$ the following inequality holds: 
        \begin{align}\label{equ:def-spectral-balance}
            \left| \lambda_{\alpha, Q}(n, \mathfrak{H}) - {\mathrm{inj}(n,Q,\mathcal{F})}/{n^{q/\alpha}}\right|
            \le \delta n^{q - q/\alpha - 1}.
        \end{align}
    \end{enumerate}
\end{definition}

The main result of this work is as follows. 
\begin{theorem}\label{THM:generalized-spectral}
    Let $q \ge r \ge 2$ be integers and $\alpha > 1$, $\varepsilon > 0$ be real numbers.
    Let $Q$ be an $r$-graph on $q$ vertices, $\mathcal{F}$ be a family of $r$-graphs with $\hat{\pi}(Q, \mathcal{F}) > 0$, and $\mathfrak{H}$ be a hereditary family of $\mathcal{F}$-free $r$-graphs. 
    There exist constants $\delta > 0$, $M$, and $N$ such that the following holds.
    Suppose that 
    \begin{enumerate}[label=(\roman*)]
        \item\label{THM:generalized-spectral-a} $\mathcal{F}$ is $(\delta, M, Q)$-smooth
        \item\label{THM:generalized-spectral-b} $\mathcal{F}$ is $(\varepsilon, M, Q)$-degree stable with respect to $\mathfrak{H}$, and 
        \item\label{THM:generalized-spectral-c} $\mathfrak{H}$ is $(\alpha, \delta, M, Q, \mathcal{F})$-balanced in spectral. 
    \end{enumerate}
    Then for every $\mathcal{F}$-free $r$-graph $\mathcal{H}$ on $n \ge N$ vertices, we have 
    \begin{align*}
        \lambda_{\alpha, Q}(\mathcal{H})
        \le \lambda_{\alpha, Q}(n, \mathfrak{H}), 
    \end{align*}
    with equality only if $\mathcal{H} \in \mathfrak{H}$.
    In particular, for every $n \ge N$, 
    \begin{align*}
        \mathrm{spex}_{\alpha}(n, Q, \mathcal{F})
        = \lambda_{\alpha, Q}(n, \mathfrak{H}). 
    \end{align*}
\end{theorem}

In the next section, we present some definitions and preliminary results. 
In Section~\ref{SEC:Proof-generlized-spectral}, we present the proof of Theorem~\ref{THM:generalized-spectral}. 
In Section~\ref{SEC:Proof-C5-K3}, we present the proof of Theorem~\ref{THM:C5-K3-spectral}. 
In Section~\ref{SEC:Entropy}, we introduce the entropic density of a generalized Tur\'{a}n problem and discuss its connection to the $(\alpha, Q)$-spectral radius. 
Section~\ref{SEC:remark} contains some concluding remarks. 

\section{Preliminaries}\label{SEC:Prelim}
Let $\alpha > 0$ be a real number and $\mathbf{x} = (x_1, \ldots, x_n) \in \Delta_{\alpha}^{n-1}$ be a vector. 
For any subset $I \subseteq [n]$ or ordered tuple $I \in [n]^{k}$ for some $k \in \mathbb{N}$, define 
\begin{align*}
    x_{I}
    \coloneqq \prod_{i \in I} x_i. 
\end{align*}
Let $Q$ and $\mathcal{H}$ be $r$-graphs.
Define the set $\mathrm{OPT}_{\alpha, Q}(\mathcal{H})$ of optimal vectors as
\begin{align*}
    \mathrm{OPT}_{\alpha, Q}(\mathcal{H})
    \coloneqq \left\{ (x_1, \ldots, x_n) \in \Delta_{\alpha}^{n-1} \colon P_{Q,\mathcal{H}}(x_1, \ldots, x_n) = \lambda_{\alpha, Q}(\mathcal{H}) \right\}. 
\end{align*}
We say that a vector $(x_1, \ldots, x_n) \in \mathbb{R}^{n}$ is \textbf{nonnegative} if $x_i \ge 0$ for every $i \in [n]$. 
We will use the following simple fact.
\begin{fact}\label{FACT:nonnegative-optimal-vector}
    The set $\mathrm{OPT}_{\alpha, Q}(\mathcal{H})$ contains a nonnegative vector. 
\end{fact}
\begin{proof}[Proof of Fact~\ref{FACT:nonnegative-optimal-vector}]
    Since $\Delta_{\alpha}^{n-1}$ is compact and $P_{Q, \mathcal{H}}$ is continuous, the set $\mathrm{OPT}_{\alpha, Q}(\mathcal{H})$ is nonempty. 
    Fix any vector $(x_1, \ldots, x_n) \in \mathrm{OPT}_{\alpha, Q}(\mathcal{H}) \subseteq \Delta_{\alpha}^{n-1}$. Note that the vector $(|x_1|, \ldots, |x_n|)$ also lies in $\Delta_{\alpha}^{n-1}$, and moreover, 
    \begin{align*}
        P_{Q, \mathcal{H}}(|x_1|, \ldots, |x_n|) 
        \ge P_{Q, \mathcal{H}}(x_1, \ldots, x_n). 
    \end{align*}
    It follows that $(|x_1|, \ldots, |x_n|) \in \mathrm{OPT}_{\alpha, Q}(\mathcal{H})$ as well. It is clear that this vector is nonnegative.
\end{proof}
Let $q \coloneqq v(Q)$. 
Fix an ordering $(u_1, \ldots, u_{q})$ of the vertex set of $Q$. 
We identify each map $\phi \in \mathrm{Inj}(Q,\mathcal{H})$ with the ordered $q$-tuple $\left(\phi(u_1), \ldots, \phi(u_{q})\right)$ (or with the unordered set $\left\{\phi(u_1), \ldots, \phi(u_{q})\right\}$ when the ordering is not crucial). 
Thus, we abbreviate $\phi(V(Q))$ by $\phi$. 

For a map $\phi \in \mathrm{Inj}(Q, \mathcal{H})$ and a subset $S\subseteq V(\mathcal{H})$, we write $\phi - S$ to denote, for notational convenience, the ordered tuple (resp. unordered set) obtained from $\left(\phi(u_1), \ldots, \phi(u_{q})\right)$ (resp. $\left\{\phi(u_1), \ldots, \phi(u_{q})\right\}$) by removing all entries (resp. elements) that lie in $S$. 
When $S = \{v\}$ consists of a single vertex, we simply write $\phi - v$ in place of $\phi - \{v\}$.

For a pair of vertices $(u, v) \in V(Q) \times V(\mathcal{H})$, define 
\begin{align*}
    \mathrm{Inj}(Q, \mathcal{H}, u \to v)
    \coloneqq \left\{  \phi \in \mathrm{Inj}(Q, \mathcal{H}) \colon \phi(u) = v \right\}.
\end{align*}
For every $v \in V(\mathcal{H})$, let 
\begin{align*}
    \mathrm{Inj}(Q, \mathcal{H}, v)
    \coloneqq \left\{ \phi \in \mathrm{Inj}(Q, \mathcal{H}) \colon v \in \phi \right\}
    = \bigcup_{w \in V(Q)} \mathrm{Inj}(Q, \mathcal{H}, w \to v).
\end{align*}
The (ordered) $Q$-link of $v$ in $\mathcal{H}$ is the following multiset of ordered\footnote{Here, we use $\multiset{}$ to indicate that the set is order multiset.} tuples:
\begin{align*}
    L_{Q, \mathcal{H}}^{o}(v)
    \coloneqq \multiset{\phi - v \colon  \phi \in \mathrm{Inj}(Q,\mathcal{H},v)}. 
\end{align*}
Note that the multiplicity of each element in $L_{Q, \mathcal{H}}^{o}(v)$ is at most $q$, since $\phi$ maps each of the $q$ vertices of $Q$ to $v$ in at most one coordinate.

For a subset $S \subseteq V(\mathcal{H})$, let $\mathcal{H}[S]$ denote the induced subgraph of $\mathcal{H}$ on $S$. 
We write $\mathcal{H} - S$ for the subgraph of $\mathcal{H}$ induced by $V(\mathcal{H}) \setminus S$, and when $S = \{v\}$ consists of a single vertex, we simply write $\mathcal{H} - v$.

\begin{fact}[H\"{o}lder's inequality]\label{FACT:Holder-Inequality}
    Suppose that $(x_1, \ldots, x_n)$ and $(y_1, \ldots, y_n)$ are nonnegative vectors and $p, q$ are positive integers satisfying $\frac{1}{p} + \frac{1}{q} = 1$. Then  
    \begin{align*}
        \sum_{i \in [n]} x_i y_i 
        \le \big(x_1^{p} + \cdots + x_{n}^{p} \big)^{\frac{1}{p}} \big( y_1^{q} + \cdots + y_n^{q} \big)^{\frac{1}{q}}. 
    \end{align*}
    In particular, taking $(x_1, \ldots, x_n) = (1,\ldots, 1)$ and $(p, q) = \big( \frac{\alpha}{\alpha-1} , \alpha\big)$ with $\alpha > 1$, we obtain 
    \begin{align*}
        \sum_{i \in [n]}  y_i 
        \le n^{\frac{\alpha-1}{\alpha}} \big( y_1^{\alpha} + \cdots + y_n^{\alpha} \big)^{\frac{1}{\alpha}}.  
    \end{align*}
\end{fact}

\begin{fact}\label{FACT:derivative-Lagrange-poly}
    Let $Q$ be an $r$-graph, and let $\mathcal{H}$ be an $r$-graph on vertex set $[n]$. 
    For every $i \in [n]$, the $i$-th partial derivative of $P(Q,\mathcal{H})(X_1, \ldots, X_n)$ satisfies
    \begin{align}\label{equ:derivative-link}
        \partial_{i} P_{Q,\mathcal{H}}(X_1, \ldots, X_n)
        = \sum_{S \in L_{Q, \mathcal{H}}^{o}(i)} X_{S}. 
    \end{align}
    In particular, by Fact~\ref{FACT:Holder-Inequality}, for every nonnegative vector $(x_1, \ldots, x_n) \in \mathbb{R}^{n}$, we have 
    \begin{align}\label{equ:derivative-link-holder}
        \partial_{i} P_{Q,\mathcal{H}}(x_1, \ldots, x_n)
        \le |\mathrm{Inj}(Q,\mathcal{H},i)|^{\frac{\alpha-1}{\alpha}} \left( \sum_{S \in L_{Q, \mathcal{H}}^{o}(i)} X_{S}^{\alpha} \right)^{1/\alpha}.
    \end{align}
\end{fact}

The following fact follows directly from the definition of $\hat{\pi}(Q,\mathcal{F})$. 
\begin{fact}\label{FACT:concentration-generalized-Turan-denisty}
    Let $Q$ be an $r$-graph and $\mathcal{F}$ be a family of $r$-graphs. 
    For every $\delta>0$, there exists a constant $N_{\ref{FACT:concentration-generalized-Turan-denisty}} = N_{\ref{FACT:concentration-generalized-Turan-denisty}}(\delta)$ such that for every $n \ge N_{\ref{FACT:concentration-generalized-Turan-denisty}}$, 
    \begin{align}\label{equ:concentration-generalized-Turan-denisty}
        \left| \mathrm{inj}(n,Q,\mathcal{F}) - \hat{\pi}(Q,\mathcal{F}) n^{q} \right|
        \le \delta n^{q}. 
    \end{align}
\end{fact}

The following fact is an immediate consequence of~\eqref{equ:def-spectral-balance} and Fact~\ref{FACT:concentration-generalized-Turan-denisty}. 
\begin{fact}\label{FACT:mu-n-lower-bound}
    Let $Q$ be an $r$-graph, $\mathcal{F}$ be a family of $r$-graphs, and $\mathfrak{H}$ be a hereditary family of $\mathcal{F}$-free $r$-graphs. 
    Suppose that $\mathfrak{H}$ is $(\alpha, \delta, N, Q, \mathcal{F})$-balanced in spectral. Then for every $\delta' > 0$, there exists a constant $N_{\ref{FACT:mu-n-lower-bound}} = N_{\ref{FACT:mu-n-lower-bound}}(\delta')$ such that for every $n \ge N_{\ref{FACT:mu-n-lower-bound}}$, 
    \begin{align*}
        & \quad \left|\lambda_{\alpha, Q}(n, \mathfrak{H}) - \hat{\pi}(Q,\mathcal{F}) n^{q - \frac{q}{\alpha}}\right| \\
        & \le \left|\lambda_{\alpha, Q}(n, \mathfrak{H}) - \frac{\mathrm{inj}(n,Q,\mathcal{F})}{n^{q/\alpha}} \right| + \left|\frac{\mathrm{inj}(n,Q,\mathcal{F})}{n^{q/\alpha}} - \hat{\pi}(Q,\mathcal{F}) n^{q - \frac{q}{\alpha}} \right| \\
        & \le \frac{\delta}{n} n^{q-\frac{q}{\alpha}} + \frac{\delta'}{2} n^{q - \frac{q}{\alpha}}
        \le \delta' n^{q - \frac{q}{\alpha}}. 
    \end{align*}
\end{fact}

\section{Proof of Theorem~\ref{THM:generalized-spectral}}\label{SEC:Proof-generlized-spectral}
In this section, we present the proof of Theorem~\ref{THM:generalized-spectral}.
Throughout, we assume that 
\begin{enumerate}[label=(\roman*)]
    \item\label{Assume-a} $r \ge 2$ is an integer and $\alpha > 1$ is a real number; 
    \item\label{Assume-b} $Q$ is an $r$-graph on the vertex set $[q]$, $\mathcal{F}$ is a family of $r$-graphs satisfying $\hat{\pi}(Q,\mathcal{F}) > 0$, $\mathfrak{H}$ is a hereditary family of $\mathcal{F}$-free $r$-graphs, and $\mathcal{H}$ is an $\mathcal{F}$-free $r$-graph on $[n]$;
    \item\label{Assume-c} $\mathbf{x} = (x_1, \ldots, x_n) \in \mathrm{OPT}_{\alpha, Q}(\mathcal{H})$ is a nonnegative optimal vector, whose existence is guaranteed by Fact~\ref{FACT:nonnegative-optimal-vector}..
\end{enumerate}
For convenience, let 
\begin{align*}
    \hat{\pi} \coloneqq \hat{\pi}(Q,\mathcal{F})
    \quad\text{and}\quad 
    \mu_{n}
    \coloneqq \lambda_{\alpha, Q}(n, \mathfrak{H}). 
\end{align*}
%

\subsection{Preparations}\label{SUBSEC:Preparation-Proof-generlized-spectral}
The following lemma is a straightforward consequence of the Lagrange multiplier method. 
\begin{lemma}\label{LEMMA:Lagrange-multiplier}
    For every $i \in [n]$, we have 
    \begin{align*}
        \partial_{i} P_{Q,\mathcal{H}} (x_1, \ldots, x_n)
        = q \cdot \lambda_{\alpha, Q}(\mathcal{H}) \cdot x_i^{\alpha-1}. 
    \end{align*}
\end{lemma}

We remark that if the nonnegativity restriction on $\mathbf{x}$ is removed, then Lemma~\ref{LEMMA:Lagrange-multiplier} becomes
\begin{align*}
    \partial_{i} P_{Q,\mathcal{H}} (x_1, \ldots, x_n)
    = q \cdot \lambda_{\alpha, Q}(\mathcal{H}) \cdot x_i |x_{i}|^{\alpha-2}. 
\end{align*}

\begin{proof}[Proof of Lemma~\ref{LEMMA:Lagrange-multiplier}]
    Define 
    \begin{align*}
        \mathcal{L}(X_1, \ldots, X_n, \beta)
        \coloneqq P_{Q, \mathcal{H}}(X_1, \ldots, X_n) - \beta  \left(|X_{1}|^{\alpha} + \cdots + |X_{n}|^{\alpha} - 1\right). 
    \end{align*}
    Since $\mathbf{x} = (x_1, \ldots, x_n) \in \mathrm{OPT}_{\alpha, Q}(\mathcal{H})$, it follows from the Lagrange multiplier method that $\partial_{i} \mathcal{L}(\mathbf{x}) = 0$ for every $i \in [n]$.
    Consequently,
    \begin{align}\label{equ:derivative-Lagrange}
        \partial_{i} P_{Q, \mathcal{H}}(\mathbf{x}) 
        = \beta \alpha x_i |x_{i}|^{\alpha-2} 
        = \beta \alpha x_i^{\alpha-1} 
        \quad\text{for every}\quad i \in [n]. 
    \end{align}
    Multiply $x_i$ on both sides and then summing over all $i \in [n]$, we obtain 
    \begin{align}\label{equ:Lagrange-multiplier-sum}
        \sum_{i \in [n]} x_i \cdot \partial_{i} P_{Q, \mathcal{H}}(\mathbf{x}) 
        = \sum_{i \in [n]} \beta \alpha x_i^{\alpha}
        = \beta \alpha.
    \end{align}
    Since 
    \begin{align*}
        \sum_{i \in [n]} x_i \cdot \partial_{i} P_{Q, \mathcal{H}}(\mathbf{x}) 
        = q \cdot P_{Q, \mathcal{H}}(\mathbf{x}),
    \end{align*}
    it follows from~\eqref{equ:Lagrange-multiplier-sum} that 
    \begin{align*}
        \beta \alpha
        = q \cdot P_{Q, \mathcal{H}}(\mathbf{x})
        = q\cdot \lambda_{\alpha,Q}(\mathcal{H}). 
    \end{align*}
    Combining it with~\eqref{equ:derivative-Lagrange}, we obtain 
    \begin{align*}
        \partial_{i} P_{Q, \mathcal{H}}(\mathbf{x})
        = q\cdot \lambda_{\alpha,Q}(\mathcal{H}) \cdot x_i^{\alpha-1}
        \quad\text{for every}\quad i \in [n], 
    \end{align*}
    which proves Lemma~\ref{LEMMA:Lagrange-multiplier}. 
\end{proof}

The following lemma provides an estimate for the difference between $\mu_n$ and $\mu_{n-1}$.
\begin{lemma}\label{LEMMA:spectral-difference}
    Suppose that $\mathcal{F}$ is $(\delta, N, Q)$-smooth and $\mathfrak{H}$ is $(\alpha, \delta, N, Q, \mathcal{F})$-balanced in spectral. 
    Then there exists $N_{\ref{LEMMA:spectral-difference}}$ such that for every $n \ge N_{\ref{LEMMA:spectral-difference}}$,
    \begin{align}\label{equ:Assumption-spectral-smooth}
        \left| \mu_{n} - \mu_{n-1} - \frac{q(\alpha-1)}{\alpha} \hat{\pi} n^{q-\frac{q}{\alpha}-1}\right|
        \le 5 \delta n^{q-\frac{q}{\alpha}-1}. 
    \end{align}
\end{lemma}
\begin{proof}[Proof of Lemma~\ref{LEMMA:spectral-difference}]
    Let $n$ be a sufficiently large integer. 
    Using~\eqref{equ:Turan-smoothness} and the inequality (see Fact~\ref{FACT:inequality-c})
    \begin{align*}
        \left(\frac{n}{n-1}\right)^{q/\alpha} 
        \ge 1 + \frac{q}{\alpha} n^{-1}, 
    \end{align*}
    we obtain 
    \begin{align*}
        & \quad \mathrm{inj}(n-1,Q,\mathcal{F}) \left(\frac{n}{n-1}\right)^{q/\alpha} \\
        & \ge \left(\mathrm{inj}(n,Q,\mathcal{F}) - q \hat{\pi}(Q,\mathcal{F}) n^{q-1} - \delta n^{q-1}\right) \left(1 + \frac{q}{\alpha} n^{-1} \right) \\
        & = \mathrm{inj}(n,Q,\mathcal{F}) + \frac{q}{\alpha} \frac{\mathrm{inj}(n,Q,\mathcal{F})}{n} - q \hat{\pi} n^{q-1} -\delta n^{q-1} - \left(q \hat{\pi}+\delta\right) \frac{q}{\alpha} \frac{1}{n} n^{q-1} \\
        & \ge \mathrm{inj}(n,Q,\mathcal{F}) + \frac{q}{\alpha} \frac{\mathrm{inj}(n,Q,\mathcal{F})}{n} - q \hat{\pi} n^{q-1} - 2 \delta n^{q-1}. 
    \end{align*}
    By choosing $n$ sufficiently large, we can ensure that~\eqref{equ:concentration-generalized-Turan-denisty} holds with $\delta$ there replaced by $\alpha \delta/q$ here. Combining this with the inequality above, we obtain
    \begin{align*}
        & \quad \left|\mathrm{inj}(n,Q,\mathcal{F}) - \mathrm{inj}(n-1,Q,\mathcal{F}) \left(\frac{n}{n-1}\right)^{q/\alpha} - \frac{q(\alpha-1)}{\alpha} \hat{\pi} n^{q-1}\right| \\
        & \le \left|- \frac{q}{\alpha} \frac{\mathrm{inj}(n,Q,\mathcal{F})}{n} + q \hat{\pi} n^{q-1} - \frac{q(\alpha-1)}{\alpha} \hat{\pi} n^{q-1}\right| + 2\delta n^{q-1} \\
        & = \left|  \hat{\pi} - \frac{\mathrm{inj}(n,Q,\mathcal{F})}{n^{q}}  \right| \frac{q}{\alpha} n^{q-1} + 2\delta n^{q-1} 
        \le \delta n^{q-1} + 2\delta n^{q-1} 
        = 3\delta n^{q-1}. 
    \end{align*}
    Combining this inequality with~\eqref{equ:def-spectral-balance}, we obtain 
    \begin{align*}
        & \quad \left| \mu_{n} - \mu_{n-1} - \frac{q(\alpha-1)}{\alpha} \hat{\pi} n^{q-\frac{q}{\alpha}-1}\right| \\
        & \le \left|\frac{\mathrm{inj}(n,Q,\mathcal{F})}{n^{q/\alpha}} - \frac{\mathrm{inj}(n-1,Q,\mathcal{F})}{(n-1)^{q/\alpha}} - \frac{q(\alpha-1)}{\alpha} \hat{\pi} n^{q-\frac{q}{\alpha}-1}\right|  + 2 \delta n^{q - \frac{q}{\alpha} - 1} \\
        & = \left|\mathrm{inj}(n,Q,\mathcal{F}) - \mathrm{inj}(n-1,Q,\mathcal{F}) \left(\frac{n}{n-1}\right)^{\frac{q}{\alpha}} - \frac{q(\alpha-1)}{\alpha} \hat{\pi} n^{q-1}\right|\frac{1}{n^{q/\alpha}}  + 2 \delta n^{q-\frac{q}{\alpha} - 1} \\
        & \le 3 \delta n^{q - \frac{q}{\alpha} -1} + 2\delta n^{q-\frac{q}{\alpha} - 1} 
        = 5\delta n^{q-\frac{q}{\alpha} - 1}, 
    \end{align*}
    which proves Lemma~\ref{LEMMA:spectral-difference}. 
\end{proof}

The following lemma shows that if $\mathcal{H}$ has a large $(\alpha, Q)$-spectral radius but contains a vertex with small degree, then the optimal vector $\mathbf{x}$ must have a coordinate with a small value.
\begin{lemma}\label{LEMMA:minimum-degree-weight}
    Let $\varepsilon \in [0,1]$ be real numbers. 
    There exists $N_{\ref{LEMMA:minimum-degree-weight}} = N_{\ref{LEMMA:minimum-degree-weight}}(\varepsilon)$ such that the following holds for every $n \ge N_{\ref{LEMMA:minimum-degree-weight}}$. 
    Suppose that $\mathcal{H}$ satisfies 
    \begin{align*}
        \lambda_{\alpha, Q}(\mathcal{H})\ge \mu_n
        \quad\text{and}\quad 
        \delta_{Q}(\mathcal{H})
        \le (1-\varepsilon) q \hat{\pi}  n^{q-1}. 
    \end{align*}
    Then there exists a coordinate $i \in [n]$ such that 
    \begin{align*}
        x_{i} 
        \le \left( 1-\frac{\hat{\pi} \varepsilon}{(q-1)\alpha} \right) n^{-\frac{1}{\alpha}}. 
    \end{align*}
\end{lemma}
\begin{proof}[Proof of Lemma~\ref{LEMMA:minimum-degree-weight}]
    Fix $\varepsilon \in [0, 1]$, and let $\delta \coloneqq \frac{\hat{\pi} \varepsilon}{(q-1)\alpha}$.
    Let $n$ be sufficiently large and suppose that $\mathcal{H}$ satisfies the assumptions in Lemma~\ref{LEMMA:minimum-degree-weight}. 
    Let $i_{\ast} \in V(\mathcal{H})$ be a vertex with 
    \begin{align}\label{equ:min-Q-degree-upper-bound}
        d_{Q,\mathcal{H}}(i_{\ast}) 
        \le (1-\varepsilon) q \hat{\pi} n^{q-1}. 
    \end{align}
    %
    Suppose to the contrary that 
    \begin{align}\label{equ:weight-lower-bound}
        x_{j} 
        \ge (1-\delta) n^{-1/\alpha}
        \quad\text{for every}\quad j \in [n]. 
    \end{align}
    \begin{claim}\label{CLAIM:Lemma-link-sum-upper-bound}
        We have $\sum_{S \in L_{Q,\mathcal{H}}^{o}(i_{\ast})} x_{S}^{\alpha} \le q  \hat{\pi}$.
    \end{claim}
    \begin{proof}[Proof of Claim~\ref{CLAIM:Lemma-link-sum-upper-bound}]
        Let $V^{q-1}$ denote the set of all tuples of length $q-1$ with entries in $V$, and let $q\cdot V^{q-1}$ denote the multiset in which each element of $V^{q-1}$ appears exactly $q$ times. 
        Let 
        \begin{align*}
            \mathcal{S} \coloneqq L_{Q,\mathcal{H}}^{o}(i_{\ast})
            \quad\text{and}\quad 
            \bar{\mathcal{S}} \coloneqq q\cdot V^{q-1} \setminus L_{Q,\mathcal{H}}^{o}(i_{\ast}). 
        \end{align*}
        Note that 
        \begin{align}\label{equ:sum-q}
            \sum_{S \in q \cdot V^{q-1}} x_{S}^{\alpha}
            = q\left(x_{1}^{\alpha} + \cdots + x_{n}^{\alpha}\right)^{q-1} 
            = q. 
        \end{align}
        Using~\eqref{equ:min-Q-degree-upper-bound},~\eqref{equ:weight-lower-bound}, and Bernoulli's inequality $(1-\delta)^{(q-1)\alpha} \ge 1- (q-1)\alpha \delta$, we obtain 
        \begin{align*}
            \sum_{S \in \bar{\mathcal{S}}} x_{S}^{\alpha} 
            & \ge \left(q n^{q-1} - (1-\varepsilon) q  \hat{\pi}  n^{q-1} \right) \left(\frac{1-\delta}{n^{1/\alpha}}\right)^{(q-1)\alpha} \\
            & = q (1-\delta)^{(q-1)\alpha} - (1-\varepsilon) q  \hat{\pi} (1-\delta)^{(q-1)\alpha} 
            \ge q - q(q-1)\alpha \delta  - (1-\varepsilon) q  \hat{\pi}.  
        \end{align*}
        Combining it with~\eqref{equ:sum-q}, we obtain 
        \begin{align*}
            \sum_{S \in \mathcal{S}} x_{S}^{\alpha}
            = \sum_{S \in q \cdot V^{q-1}} x_{S}^{\alpha} - \sum_{S \in \bar{\mathcal{S}}} x_{S}^{\alpha} 
            & \le q - \big( q - q(q-1)\alpha \delta  - (1-\varepsilon) q  \hat{\pi} \big) \\
            & = q  \hat{\pi} + q(q-1)\alpha \delta - \varepsilon q  \hat{\pi}
            = q  \hat{\pi},
        \end{align*}
        which proves Claim~\ref{CLAIM:Lemma-link-sum-upper-bound}. 
    \end{proof}

    Combining Claim~\ref{CLAIM:Lemma-link-sum-upper-bound} with Lemma~\ref{LEMMA:Lagrange-multiplier} and Fact~\ref{FACT:derivative-Lagrange-poly}, we obtain 
    \begin{align*}
        \lambda_{\alpha,Q}(\mathcal{H}) \cdot x_{i_{\ast}}^{\alpha-1}
        = \frac{\partial_{i_{\ast}} P_{Q, \mathcal{H}}(\mathbf{x})}{q} 
        & \le \frac{\big(d_{Q,\mathcal{H}}(i_{\ast})\big)^{\frac{\alpha-1}{\alpha}}}{q} \left(\sum_{S \in \mathcal{S}} x_{S}^{\alpha}\right)^{1/\alpha} \\
        & \le \frac{1}{q} \left((1-\varepsilon)q \hat{\pi} n^{q-1}\right)^{\frac{\alpha-1}{\alpha}} \left(q \hat{\pi}\right)^{\frac{1}{\alpha}} \\
        & =  (1-\varepsilon)^{\frac{\alpha-1}{\alpha}} \hat{\pi} n^{(q-1)\frac{\alpha-1}{\alpha}}
        \le (1-\varepsilon) \hat{\pi} n^{(q-1)\frac{\alpha-1}{\alpha}}.
    \end{align*}

    On the other hand, combining the assumption $\lambda_{\alpha,Q}(\mathcal{H}) \ge \mu_{n}$ with~\eqref{equ:weight-lower-bound} and Fact~\ref{FACT:mu-n-lower-bound}, we obtain 
    \begin{align*}
        \lambda_{\alpha,Q}(\mathcal{H}) \cdot x_{i_{\ast}}^{\alpha-1}
        \ge \mu_{n} x_{i_{\ast}}^{\alpha-1}
        & \ge (1-\delta) \hat{\pi}  n^{q - \frac{q}{\alpha}} \left(\frac{1-\delta}{n^{1/\alpha}}\right)^{\alpha-1} \\
        & = (1-\delta)^{\alpha} \hat{\pi} n^{(q-1)\frac{\alpha-1}{\alpha}} \\
        & \ge (1-\alpha\delta) \hat{\pi} n^{(q-1)\frac{\alpha-1}{\alpha}}
        > (1- \varepsilon) \hat{\pi} n^{(q-1)\frac{\alpha-1}{\alpha}}, 
    \end{align*}
    which contradicts the preceding inequality. 
    This completes the proof of Lemma~\ref{LEMMA:minimum-degree-weight}. 
\end{proof}

The following lemma provides a lower bound on the $(\alpha, Q)$-spectral radius of the induced subgraph obtained by removing a vertex whose weight is small in the optimal vector $\mathbf{x}$.
\begin{lemma}\label{LEMMA:delete-light-vertex}
    Let $\delta \in (0,1)$ be a real number and $N \ge 1$ be an integer. 
    There exists $N_{\ref{LEMMA:delete-light-vertex}} = N_{\ref{LEMMA:delete-light-vertex}}(\delta, N)$ such that the following holds for all $n \ge N_{\ref{LEMMA:delete-light-vertex}}$. 
    Suppose that the nonnegative optimal vector $\mathbf{x}$ contains a coordinate $x_i$ satisfying $x_{i} \le (1-\delta)n^{-1/\alpha}$. 
    Then the induced subgraph $\mathcal{H}-i$ satisfies 
    \begin{align}\label{equ:LEMMA:delete-light-vertex-a}
        \lambda_{\alpha, Q}\left( \mathcal{H}-i \right)
        \ge \left(1-\frac{q(\alpha-1)}{\alpha}\left(1-\frac{\delta}{2}\right)\frac{1}{n}\right) \lambda_{\alpha, Q }(\mathcal{H}).
    \end{align}
    Suppose further that $\mathcal{F}$ is $(\xi/5, N, Q)$-smooth, $\mathfrak{H}$ is $(\alpha, \xi/5, N, Q, \mathcal{F})$-balanced in spectral for some constant $\xi < \frac{q(\alpha-1) \hat{\pi} \delta}{4 \alpha}$, and $\lambda_{\alpha, Q}(\mathcal{H}) \ge \mu_{n}$.  Then
    \begin{align}\label{equ:LEMMA:delete-light-vertex-b}
        \lambda_{\alpha, Q}\left( \mathcal{H}-i \right) 
        > \mu_{n-1}. 
    \end{align}
\end{lemma}
\begin{proof}[Proof of Lemma~\ref{LEMMA:delete-light-vertex}]
    Let $n$ be sufficiently large. 
    By relabeling the vertices of $\mathcal{H}$ if necessary, we may assume that $x_n \le (1-\delta) n^{-1/\alpha}$. 
    Let $\mathcal{G}$ denote the induced subgraph of $\mathcal{H}$ on $[n-1]$. 
    Let 
    \begin{align*}
        \mathbf{y} 
        \coloneqq  \left(\frac{x_1}{\left(1-x_n^{\alpha}\right)^{1/\alpha}}, \ldots, \frac{x_{n-1}}{\left(1-x_n^{\alpha}\right)^{1/\alpha}} \right) 
        \in \Delta_{n-2}^{\alpha}. 
    \end{align*}
    By Lemma~\ref{LEMMA:Lagrange-multiplier}, we have 
    \begin{align*}
        \lambda_{\alpha, Q}(\mathcal{H})
        = P_{Q,\mathcal{H}}(x_1, \ldots, x_n)
        & = P_{Q,\mathcal{G}}(x_1, \ldots, x_{n-1}) + x_{n} \cdot \partial_{n}P_{Q,\mathcal{H}}(x_1, \ldots, x_n) \\
        & = P_{Q,\mathcal{G}}(x_1, \ldots, x_{n-1}) + q \cdot \lambda_{\alpha,Q}(\mathcal{H}) \cdot x_{n}^{\alpha}, 
    \end{align*}
    which implies that 
    \begin{align}\label{equ:PG-PH}
        P_{Q,\mathcal{G}}(x_1, \ldots, x_{n-1})
        = \left( 1- q x_{n}^{\alpha} \right) \lambda_{\alpha, Q}(\mathcal{H}). 
    \end{align}
    On the other hand, we have 
    \begin{align*}
        P_{Q,\mathcal{G}}(x_1, \ldots, x_{n-1})
        & = \left(1-x_n^{\alpha}\right)^{q/\alpha} \cdot P_{Q,\mathcal{G}}\left(\mathbf{y}\right)
        \le \left(1-x_n^{\alpha}\right)^{q/\alpha} \cdot \lambda_{\alpha, Q}(\mathcal{G}).
    \end{align*}
    Combining it with~\eqref{equ:PG-PH}, we obtain 
    \begin{align*}
        \lambda_{\alpha, Q}(\mathcal{G})
        & \ge \left(1-x_n^{\alpha}\right)^{-q/\alpha} \cdot P_{Q,\mathcal{G}}(x_1, \ldots, x_{n-1}) 
        = \left(1-x_n^{\alpha}\right)^{-q/\alpha} \left( 1- q x_{n}^{\alpha} \right) \lambda_{\alpha, Q}(\mathcal{H}). 
    \end{align*}

    Since $\left(1-x_n^{\alpha}\right)^{-q/\alpha} \ge 1 + q x_{n}^{\alpha}/\alpha$ (see Fact~\ref{FACT:inequality-a}) and $x_n \le (1-\delta)n^{-1/\alpha}$, it follows from the inequality above that 
    \begin{align}\label{equ:Lemma-lambda-G-lower-bound}
        \lambda_{\alpha, Q}(\mathcal{G})
        & \ge \left(1+\frac{q}{\alpha} x_{n}^{\alpha}\right) \left( 1- q x_{n}^{\alpha} \right) \lambda_{\alpha, Q}(\mathcal{H}) \notag \\
        & = \left(1 - \frac{q(\alpha-1)}{\alpha}x_{n}^{\alpha} - \frac{q^2}{\alpha} x_{n}^{2 \alpha}\right) \lambda_{\alpha, Q}(\mathcal{H}) \notag \\
        & \ge \left(1 - \frac{q(\alpha-1)}{\alpha} \left( \frac{1-\delta}{n^{1/\alpha}} \right)^{\alpha}  - \frac{q^2}{\alpha} \left( \frac{1-\delta}{n^{1/\alpha}} \right)^{2 \alpha} \right) \lambda_{\alpha, Q}(\mathcal{H}) \notag \\
        & \ge \left(1 - \frac{q(\alpha-1)}{\alpha} \frac{1-\delta}{n}  - \frac{q^2}{\alpha} \frac{1}{n^2} \right) \lambda_{\alpha, Q}(\mathcal{H}) \notag \\
        & \ge \left(1 - \frac{q(\alpha-1)}{\alpha} \left( 1-\frac{\delta}{2} \right) \frac{1}{n} \right) \lambda_{\alpha, Q}(\mathcal{H}), 
    \end{align}
    which proves~\eqref{equ:LEMMA:delete-light-vertex-a}.

    Now suppose that $\mathcal{F}$ is $(\xi/5, N, Q)$-smooth, $\mathfrak{H}$ is $(\alpha, \xi/5, N, Q, \mathcal{F})$-balanced in spectral for some constant $\xi < \frac{q(\alpha-1) \hat{\pi} \delta}{4 \alpha}$, and $\lambda_{\alpha, Q}(\mathcal{H}) \ge \mu_{n}$. 
    It follows from Lemma~\ref{LEMMA:spectral-difference} that 
    \begin{align*}
        \lambda_{\alpha, Q}(\mathcal{H}) 
        \ge \mu_{n}
        \ge \mu_{n-1} + \left( \frac{q(\alpha-1)}{\alpha} \hat{\pi} - \xi \right)n^{q-\frac{q}{\alpha}-1}.
    \end{align*}
    Combining it with~\eqref{equ:Lemma-lambda-G-lower-bound}, we obtain 
    \begin{align}\label{equ:lambda-G-lower-bound}
        \lambda_{\alpha, Q}(\mathcal{G})
        & \ge \left(1 - \frac{q(\alpha-1)}{\alpha} \left( 1-\frac{\delta}{2} \right) \frac{1}{n} \right) \left(\mu_{n-1} + \left( \frac{q(\alpha-1)}{\alpha} \hat{\pi} - \xi \right)n^{q-\frac{q}{\alpha}-1} \right) \notag \\
        & = \mu_{n-1}  - \frac{q(\alpha-1)}{\alpha} \left( 1-\frac{\delta}{2} \right) \frac{\mu_{n-1}}{n}  + \left( \frac{q(\alpha-1)}{\alpha} \hat{\pi} - \xi \right)n^{q-\frac{q}{\alpha}-1}  \notag \\
        & \quad - \frac{q(\alpha-1)}{\alpha} \left( 1-\frac{\delta}{2} \right) \left( \frac{q(\alpha-1)}{\alpha} \hat{\pi} - \xi \right)n^{q-\frac{q}{\alpha}-2}. 
    \end{align}
    By Fact~\ref{FACT:mu-n-lower-bound}, we can assume from the beginning that $n$ is sufficiently large so that
    \begin{align*}
        \mu_{n-1} 
        \le \left(1+\frac{\delta}{5}\right) \hat{\pi} (n-1)^{q-\frac{q}{\alpha}} 
        \le \left(1+\frac{\delta}{5}\right) \hat{\pi} n^{q-\frac{q}{\alpha}},
    \end{align*}
    Combining it with the inequality $-(1-\delta/2)(1+\delta/5) \ge -(1-\delta/4)$, we obtain 
    \begin{align*}
        & \quad - \frac{q(\alpha-1)}{\alpha} \left( 1-\frac{\delta}{2} \right) \frac{\mu_{n-1}}{n}  + \left( \frac{q(\alpha-1)}{\alpha} \hat{\pi} - \xi \right)n^{q-\frac{q}{\alpha}-1} \\
        & \ge - \frac{q(\alpha-1)}{\alpha} \left( 1-\frac{\delta}{2} \right) \left(1+\frac{\delta}{5}\right) \hat{\pi} n^{q-\frac{q}{\alpha}}  + \left( \frac{q(\alpha-1)}{\alpha} \hat{\pi} - \xi \right)n^{q-\frac{q}{\alpha}-1}  \\
        & \ge - \frac{q(\alpha-1)}{\alpha} \left(1-\frac{\delta}{4}\right) \hat{\pi} n^{q-\frac{q}{\alpha} -1}  + \left( \frac{q(\alpha-1)}{\alpha} \hat{\pi} - \xi \right) n^{q-\frac{q}{\alpha}-1}   \\
        & = \left(\frac{q(\alpha-1) \delta}{4 \alpha} \hat{\pi} - \xi \right) n^{q-\frac{q}{\alpha}-1} 
        > \frac{q(\alpha-1)}{\alpha} \left( 1-\frac{\delta}{2} \right) \left( \frac{q(\alpha-1)}{\alpha} \hat{\pi} - \xi \right)\frac{1}{n} n^{q-\frac{q}{\alpha}-1}.
    \end{align*}
    Combining this inequality with~\eqref{equ:lambda-G-lower-bound}, we obtain that $\lambda_{\alpha, Q}(\mathcal{G}) > \mu_{n-1}$.
    This proves~\eqref{equ:LEMMA:delete-light-vertex-b}. 
\end{proof}

\subsection{Proof of Theorem~\ref{THM:generalized-spectral}}\label{SUBSEC:Proof-generlized-spectral}
We are now ready to present the proof of Theorem~\ref{THM:generalized-spectral}.
\begin{proof}[Proof of Theorem~\ref{THM:generalized-spectral}] 
    Fix $\varepsilon > 0$. Let $\delta \coloneqq  \frac{\hat{\pi} \varepsilon}{(q-1)\alpha}$ and $\xi \coloneqq \frac{q(\alpha-1) \hat{\pi} \delta}{5 \alpha} < \frac{q(\alpha-1) \hat{\pi} \delta}{4 \alpha}$. 
    Suppose that $M$ is an integer satisfying 
    \begin{enumerate}[label=(\roman*)]
        \item\label{Assump:proof-main-a} $\mathcal{F}$ is $(\xi/5, M, Q)$-smooth,
        \item\label{Assump:proof-main-b} $\mathcal{F}$ is $(\varepsilon, M, Q)$-degree stable with respect to $\mathfrak{H}$, and 
        \item\label{Assump:proof-main-c} $\mathfrak{H}$ is $(\alpha, \xi/5, M, Q, \mathcal{F})$-balanced in spectral. 
    \end{enumerate}
    Fix a sufficiently large integer $N$ so that $N \ge \max\left\{M,~N_{\ref{LEMMA:minimum-degree-weight}}(\varepsilon),~N_{\ref{LEMMA:delete-light-vertex}}(\delta, M)\right\}$, where $N_{\ref{LEMMA:minimum-degree-weight}}(\varepsilon)$ and $N_{\ref{LEMMA:delete-light-vertex}}(\delta, M)$ are constants given by Lemmas~\ref{LEMMA:minimum-degree-weight} and~\ref{LEMMA:delete-light-vertex}. 
    Let $n$ be a sufficiently large integer so that, in particular, 
    \begin{align}\label{equ:n-lower-bound}
        \hat{\pi} e^{-3q^2} n^{\frac{\delta q (\alpha-1)}{2\alpha}} > 2 N^{q}.
    \end{align}
    Suppose to the contrary that there exists an $\mathcal{F}$-free $r$-graph $\mathcal{H}$ on $[n]$ such that 
    \begin{align*}
        \lambda_{\alpha,Q}(\mathcal{H})
        \ge \lambda_{\alpha, Q}(n,\mathfrak{H})
        \quad\text{but}\quad 
        \mathcal{H} \not\in \mathfrak{H}.
    \end{align*}
    Then it follows from Assumption~\ref{Assump:proof-main-b} that 
    \begin{align}\label{equ:min-Q-degree-Hn}
        \delta_{Q}(\mathcal{H})
        \le (1-\varepsilon) q \hat{\pi}  n^{q-1}. 
    \end{align}
    Recall that $\mathbf{x} = (x_1, \ldots, x_{n}) \in \mathrm{OPT}_{\alpha, Q}(\mathcal{H})$ is a nonnegative optimal vector. 
    Let $\mathcal{H}_{n} \coloneqq \mathcal{H}$. 
    It follows from~\eqref{equ:min-Q-degree-Hn} and Lemma~\ref{LEMMA:minimum-degree-weight} that there exists a coordinate $i_{\ast} \in V(\mathcal{H})$ such that 
    \begin{align}\label{equ:x-i-ast-small}
        x_{i_{\ast}} \le (1-\delta)n^{-1/\alpha}. 
    \end{align}
    Define $\mathcal{H}_{n-1} \coloneqq \mathcal{H}_{n} - i_{\ast}$. 
    It follows from~\eqref{equ:x-i-ast-small} and Lemma~\ref{LEMMA:delete-light-vertex} that 
    \begin{align*}
        \lambda_{\alpha, Q}(\mathcal{H}_{n-1})
        \ge \left(1-\frac{q(\alpha-1)}{\alpha}\left(1-\frac{\delta}{2}\right) \frac{1}{n}\right) \lambda_{\alpha, Q}(\mathcal{H}_{n}),
    \end{align*}
    Moreover, from the inequality $\lambda_{\alpha, Q}(\mathcal{H}) \ge \mu_n$ with Assumptions~\ref{Assump:proof-main-a} and~\ref{Assump:proof-main-c}, we obtain 
    \begin{align*}
        \lambda_{\alpha, Q}(\mathcal{H}_{n-1})
        > \mu_{n-1}. 
    \end{align*}
    Thus, by Assumption~\ref{Assump:proof-main-b}, we have  
    \begin{align}\label{equ:min-Q-degree-Hn-1}
        \delta_{Q}(\mathcal{H}_{n-1})
        \le (1-\varepsilon) q \hat{\pi}  (n-1)^{q-1}. 
    \end{align}
    By relabeling the vertices in $\mathcal{H}$ we may assume that $i_{\ast} = n$, and thus $V(\mathcal{H}_{n-1}) = [n-1]$. 
    Fix an arbitrary nonnegative optimal vector $\mathbf{y} = (y_1, \ldots, y_{n-1}) \in \mathrm{OPT}_{\alpha, Q}(\mathcal{H}_{n-1})$.
    It follows from~\eqref{equ:min-Q-degree-Hn-1} and Lemma~\ref{LEMMA:minimum-degree-weight} that there exists a coordinate $j_{\ast} \in V(\mathcal{H}_{n-1})$ such that 
    \begin{align}\label{equ:y-j-ast-small}
        y_{j_{\ast}} \le (1-\delta)(n-1)^{-1/\alpha}.
    \end{align}
    Define $\mathcal{H}_{n-2} \coloneqq \mathcal{H}_{n-1} - j_{\ast}$. 
    It follows from~\eqref{equ:y-j-ast-small} and Lemma~\ref{LEMMA:delete-light-vertex} that 
    \begin{align*}
        \lambda_{\alpha, Q}(\mathcal{H}_{n-2})
        \ge \left(1-\frac{q(\alpha-1)}{\alpha}\left(1-\frac{\delta}{2}\right) \frac{1}{n-1}\right) \lambda_{\alpha, Q}(\mathcal{H}_{n-1}),
    \end{align*}
    Moreover, from the inequality $\lambda_{\alpha, Q}(\mathcal{H}_{n-1}) > \mu_{n-1}$ with Assumptions~\ref{Assump:proof-main-a} and~\ref{Assump:proof-main-c}, we obtain  
    \begin{align*}
        \lambda_{\alpha, Q}(\mathcal{H}_{n-2})
        > \mu_{n-2}. 
    \end{align*}
    Therefore, we can repeat this process until we reach $\mathcal{H}_{N}$, ensuring that  
    \begin{align}\label{equ:Hk-1-Hk-spectral-a}
        \lambda_{\alpha, Q}(\mathcal{H}_{m-1})
        \ge \left(1-\frac{q(\alpha-1)}{\alpha}\left(1-\frac{\delta}{2}\right) \frac{1}{m}\right) \lambda_{\alpha, Q}(\mathcal{H}_{m})
        \quad\text{for}\quad m \in [N+1, n-1].
    \end{align}
    Using inequalities $\sum_{m=N+1}^{n} \le \ln(n/N) + 1$ and $\sum_{i=1}^{\infty} m^{-2} \le \pi^2/6 < 2$, we obtain  
    \begin{align*}
        \Phi & \coloneqq -\sum_{m=N+1}^{n} \frac{q(\alpha-1)}{\alpha}\left(1-\frac{\delta}{2}\right) \frac{1}{m} - \sum_{m=N+1}^{n} \frac{q^2 (\alpha-1)^2}{\alpha^2}\left(1-\frac{\delta}{2}\right)^2\frac{1}{m^2} \\
        & \ge - \frac{q(\alpha-1)}{\alpha}\left(1-\frac{\delta}{2}\right) \ln\left(\frac{n}{N}\right) - 3 q^2. 
    \end{align*}
    Since $1-x \ge e^{-x-x^2}$ for $x\in [0,1/2]$ (see Fact~\ref{FACT:inequality-c}), it follows that 
    \begin{align*}
        \prod_{m=N+1}^{n} \left(1-\frac{q(\alpha-1)}{\alpha}\left(1-\frac{\delta}{2}\right) \frac{1}{m}\right)
        & \ge \mathrm{exp}\left(\Phi\right) \\
        & \ge \mathrm{exp}\left(- \frac{q(\alpha-1)}{\alpha}\left(1-\frac{\delta}{2}\right) \ln\left(\frac{n}{N}\right) - 2 q^2\right) \\
        & = e^{-3q^2} \left(\frac{n}{N}\right)^{-\frac{q(\alpha-1)}{\alpha} + \frac{q(\alpha-1) \delta}{2 \alpha}}. 
    \end{align*}
    Combining this with~\eqref{equ:Hk-1-Hk-spectral-a}, Fact~\ref{FACT:mu-n-lower-bound} (with $\delta'$ there replaced by $1/2$), and assumption~\eqref{equ:n-lower-bound}, we obtain 
    \begin{align*}
        \lambda_{\alpha, Q}(\mathcal{H}_{N})
        & \ge \lambda_{\alpha, Q}(\mathcal{H}_{n}) \cdot \prod_{m=N+1}^{n} \left(1-\frac{q(\alpha-1)}{\alpha}\left(1-\frac{\delta}{2}\right) \frac{1}{m}\right) \\
        & \ge \mu_{n} \cdot e^{-3q^2} \left(\frac{n}{N}\right)^{-\frac{q(\alpha-1)}{\alpha} + \frac{q(\alpha-1) \delta}{2 \alpha}} \\
        & \ge \hat{\pi} n^{q-\frac{q}{\alpha}} \cdot e^{-3q^2} n^{-\frac{q(\alpha-1)}{\alpha} + \frac{q(\alpha-1) \delta}{2 \alpha}} N^{\frac{q(\alpha-1)}{\alpha}\left(1-\frac{\delta}{2}\right)}/2 \\
        & \ge \hat{\pi} e^{-3q^2} n^{\frac{q(\alpha-1)\delta}{2\alpha} } N^{\frac{q(\alpha-1)}{\alpha}\left(1-\frac{\delta}{2}\right)}/2
        > N^{q}, 
    \end{align*}
    which is clearly impossible. 
    This completes the proof of Theorem~\ref{THM:generalized-spectral}. 
\end{proof}

\section{Proof of Theorem~\ref{THM:C5-K3-spectral}}\label{SEC:Proof-C5-K3}
In this section, we present the proof of Theorem~\ref{THM:C5-K3-spectral}.

Let $m \ge q \ge 1$ be integers. 
For $n \ge 1$, denote by $T_{m,n}^{q}$ the balanced complete $m$-partite $q$-graph on $n$ vertices. 
That is, let $V_1 \cup \cdots \cup V_{m} = [n]$ be a partition such that $|V_1| \ge \cdots \ge |V_{m}| \ge |V_{1}|-1$. Then $T_{m,n}^{q}$ is the $q$-graph consisting of all $q$-subsets of $[n]$ that intersect each part $V_i$ in at most one vertex. 

We will use the following estimate for $\lambda_{\alpha}(T_{m,n}^{q})$, the $\alpha$-spectral radius of $T_{m,n}^{q})$. 
\begin{lemma}\label{LEMMA:spectral-balanced-Turan-r-graph}
    Let $m \ge q \ge 2$ be integers and $\alpha > 1$ be a real number.
    There exists a constant $C_{\ref{LEMMA:spectral-balanced-Turan-r-graph}} = C_{\ref{LEMMA:spectral-balanced-Turan-r-graph}}(m,q,\alpha)$ such that for every $n \in \mathbb{N}$, 
    \begin{align*}
        {q! |T_{m,n}^{q}|}{n^{-\frac{q}{\alpha}}}
        \le \lambda_{\alpha}(T_{m,n}^{q}) 
        \le \left( 1+ {C_{\ref{LEMMA:spectral-balanced-Turan-r-graph}}}{n^{-2}} \right) {q! |T_{m,n}^{q}|}{n^{-\frac{q}{\alpha}}}. 
    \end{align*}
\end{lemma}

In the proof of Lemma~\ref{LEMMA:spectral-balanced-Turan-r-graph} and Theorem~\ref{THM:C5-K3-spectral}, we will use the following theorem. 
\begin{theorem}[{\cite[Theorems~2~and~3]{KNY15}}]\label{THM:KNY_spectral-upper-bound-k-partite-graph}  
    Let $\alpha > 1$ be a real number and $m \ge q \ge 2$ be integers. Suppose that $\mathcal{H}$ is an $m$-partite $q$-graph on $n$ vertices. 
    Then 
    \begin{align}
        \lambda_{\alpha}(\mathcal{H})
        & \le \lambda_{\alpha}(T_{m,n}^{q}), \quad\text{and} \label{equ:KNY15-a} \\[0.3em]
        \lambda_{\alpha}(\mathcal{H})
        & \le q!\binom{m}{q}^{\frac{1}{\alpha}}m^{-\frac{q}{\alpha}}|\mathcal{H}|^{1-\frac{1}{\alpha}}. \label{equ:KNY15-b}
    \end{align}
\end{theorem}

\begin{proof}[Proof of Lemma~\ref{LEMMA:spectral-balanced-Turan-r-graph}]
    Fix a real number $\alpha > 1$ and integers $m \ge q \ge 2$. 
    By Fact~\ref{FACT:lower-bound-spex}, it suffices to prove the upper bound in Lemma~\ref{LEMMA:spectral-balanced-Turan-r-graph}. 
    First, we prove the following estimate for $|T_{m,n}^{q}|$.
    \begin{claim}\label{CLAIM:Turan-number-concentration}
        There exists a constant $C_{\ref{CLAIM:Turan-number-concentration}} = C_{\ref{CLAIM:Turan-number-concentration}}(m,q)$ such that for every $n \in \mathbb{N}$, 
        \begin{align*}
            \left( 1- C_{\ref{CLAIM:Turan-number-concentration}} n^{-2} \right)\binom{m}{q}\left(\frac{n}{m}\right)^q 
            \le |T_{m,n}^{q}| 
            \le \binom{m}{q}\left(\frac{n}{m}\right)^q. 
        \end{align*}
    \end{claim}
    \begin{proof}[Proof of Claim~\ref{CLAIM:Turan-number-concentration}]
        The upper bound is well-known, so it suffices to prove that $|T_{m,n}^{q}| \ge \binom{m}{q}\left(\frac{n}{m}\right)^q - C_{\ref{CLAIM:Turan-number-concentration}} n^{q-2}$ for some constant $C_{\ref{CLAIM:Turan-number-concentration}}$ depending only on $m$ and $q$. 
        
        Fix $n \ge 1$. Let $(s,t) \in \mathbb{N}^2$ be integers such that $n = ms + t$ and $0 \le t \le m-1$. 
        Let 
        \begin{align*}
            f(x)
            \coloneqq \sum_{i=0}^{q} \binom{t}{i}\binom{m-t}{q-i} \left(Y + x\right)^{i} \left(Y - \frac{t x}{m-t} \right)^{q-i}, 
            \quad\text{where}\quad 
            Y \coloneqq \frac{n}{m} = \frac{ms+t}{m}. 
        \end{align*}
        Straightforward calculations show that
        \begin{align*}
            \frac{\mathrm{d} f(x)}{\mathrm{d} x}
            & = \sum_{i=1}^{q} \binom{t}{i}\binom{m-t}{q-i} i \left(Y + x\right)^{i-1} \left(Y -\frac{t x}{m-t} \right)^{q-i} \\
            & \quad + \sum_{i=0}^{q-1} \binom{t}{i}\binom{m-t}{q-i} \frac{-(q-i)t}{m-t} \left(Y + x\right)^{i} \left(Y -\frac{t x}{m-t} \right)^{q-i-1}.
        \end{align*}
        Since (as elementary calculations show) the coefficient of the $Y^{q-1}$ term in $\frac{\mathrm{d} f(x)}{\mathrm{d} x}$ is 
        \begin{align*}
            \sum_{i=1}^{q} \binom{t}{i}\binom{m-t}{q-i} i 
            - \frac{t}{m-t} \sum_{i=0}^{q-1} \binom{t}{i}\binom{m-t}{q-i} \frac{(q-i)t}{m-t}
            = 0,  
        \end{align*}
        it follows that  
        \begin{align}\label{equ:derivative-upper-bound}
            \frac{\mathrm{d} f(x)}{\mathrm{d} x}
            = \sum_{i=0}^{q-2} \alpha_i Y^{i}, 
        \end{align}
        where each $\alpha_i$ is a constant depending only on $m, q, x, i$. 
    
        Observe that
        \begin{align*}
            f(0) = \binom{m}{q}\left(\frac{n}{m}\right)^{q}
            \quad\text{and}\quad 
            f\left(\frac{m-t}{m}\right)
            = \sum_{i=0}^{q} \binom{t}{i}\binom{m-t}{q-i} \left(s+1\right)^{i} s^{q-i}
            = |T_{m,n}^{q}|. 
        \end{align*}
        By the Mean Value Theorem, there exists $y \in \left[0, \frac{m-t}{m}\right]$ such that 
        \begin{align*}
            \left| |T_{m,n}^{q}| - \binom{m}{q}\left(\frac{n}{m}\right)^{q} \right|
             = \left| f\left(\frac{m-t}{m}\right) - f(0) \right| 
             = \left| \frac{\mathrm{d} f(y)}{\mathrm{d} y} \frac{m-t}{m} \right|
            \le C Y^{q-2}, 
        \end{align*}
        where the last inequality follows from~\eqref{equ:derivative-upper-bound}, and $C$ is a constant depending only on $m$ and $q$. 
        This proves Claim~\ref{CLAIM:Turan-number-concentration}. 
    \end{proof}

    By Claim~\ref{CLAIM:Turan-number-concentration} and the assumption that $n$ is sufficiently large, we obtain 
    \begin{align*}
        |T_{m,n}^{q}|^{-\frac{1}{\alpha}}
        \le \left( 1- C_{\ref{CLAIM:Turan-number-concentration}} n^{-2} \right)^{-\frac{1}{\alpha}} \binom{m}{q}^{-\frac{1}{\alpha}} \left(\frac{n}{m}\right)^{-\frac{q}{\alpha}}  
        \le \left( 1 + \frac{2C_{\ref{CLAIM:Turan-number-concentration}}}{\alpha} n^{-2} \right) \binom{m}{q}^{-\frac{1}{\alpha}} \left(\frac{n}{m}\right)^{-\frac{q}{\alpha}},  
    \end{align*}
    Combining it with~\eqref{equ:KNY15-b}, we obtain 
    \begin{align*}
        \lambda_{\alpha}(T_{m,n}^{q}) 
        & \le q!\binom{m}{q}^{\frac{1}{\alpha}}m^{-\frac{q}{\alpha}}|T_{m,n}^{q}| \left( 1 + \frac{2C_{\ref{CLAIM:Turan-number-concentration}}}{\alpha} n^{-2} \right) \binom{m}{q}^{-\frac{1}{\alpha}} \left(\frac{n}{m}\right)^{-\frac{q}{\alpha}} \\
        & \le \left( 1 + \frac{2C_{\ref{CLAIM:Turan-number-concentration}}}{\alpha} n^{-2} \right) q! |T_{m,n}^{q}| n^{-\frac{q}{\alpha}}. 
    \end{align*}
    This completes the proof of Lemma~\ref{LEMMA:spectral-balanced-Turan-r-graph}. 
\end{proof}

We are now ready to present the proof of Theorem~\ref{THM:C5-K3-spectral}. 
\begin{proof}[Proof of Theorem~\ref{THM:C5-K3-spectral}]
    Recall that $\mathfrak{C}_{5}$ denotes the collection of all $C_5$-colorable graphs. 
    This family is clearly hereditary, and every graph in $\mathfrak{C}_{5}$ is $K_{3}$-free.
    It was shown in~{\cite[Theorem~2.3]{CL24}} that there exist $\varepsilon > 0$ and $N'$ such that $K_{3}$ is $(\varepsilon, N', C_5)$-degree stable with respect to $\mathfrak{C}_{5}$. 
    Let $\delta > 0$, $M$, and $N$ be the constants given by Theorem~\ref{THM:generalized-spectral}. 
    By enlarging $M$ and $N$ if necessary, we may assume that they are sufficiently large, and in particular that $\min\{M,~N\} \ge N'$. 
    
    By Theorem~\ref{THM:generalized-spectral}, it remains to show that $\mathcal{F}$ is $(\delta, M, Q)$-smooth and that $\mathfrak{C}_{5}$ is $(\alpha, \delta, M, Q, \mathcal{F})$-balanced in spectral. 
    Note that the former property follows easily from Theorem~\ref{THM:Pentagon-exact}, so we focus on verifying the latter.

    \begin{claim}\label{CLAIM:C5K3-spectral-balanced}
        The family $\mathfrak{C}_{5}$ is $(\alpha, \delta, M, Q, \mathcal{F})$-balanced in spectral. 
    \end{claim}
    \begin{proof}[Proof of Claim~\ref{CLAIM:C5K3-spectral-balanced}]
        Fix $n \ge M$. 
        By Fact~\ref{FACT:lower-bound-spex}, it suffices to show that 
        \begin{align}\label{equ:spectral-C5-upper-bound}
            \lambda_{\alpha, C_5}(n, \mathfrak{C}_5)
            \le \frac{\mathrm{inj}(n,C_5, K_3)}{n^{5/\alpha}}
            + \delta n^{4-5/\alpha}
            = \frac{|\mathrm{Aut}(C_5)| |T_{5,n}^{5}|}{n^{5/\alpha}} + \delta n^{4-5/\alpha},  
        \end{align}
        where the equality follows from Theorem~\ref{THM:Pentagon-exact}. 
        
        Let $\mathfrak{T}_{5}$ be the collection of all $5$-partite $5$-graphs. 
        For every $G \in \mathfrak{C}_{5}$, we associate a $5$-graph $\mathcal{H}_{G}$ defined by
        \begin{align*}
            \mathcal{H}_{G}
            \coloneqq \left\{ S\in \binom{V(G)}{5} \colon C_5 \subseteq G[S] \right\}. 
        \end{align*}
        Observe that $\mathcal{H}_{G} \in \mathfrak{T}_{5}$ and 
        \begin{align*}
            P_{C_5, G}(X_1, \ldots, X_n)
            = \frac{|\mathrm{Aut}(C_5)|}{5!} P_{\mathcal{H}_{G}}(X_1, \ldots, X_n). 
        \end{align*}
        So it follows from Lemma~\ref{LEMMA:spectral-balanced-Turan-r-graph} and~\eqref{equ:KNY15-a} that 
        \begin{align*}
            \lambda_{\alpha, C_5}(n, \mathfrak{C}_5)
            \le \frac{|\mathrm{Aut}(C_5)|}{5!} \lambda_{\alpha}(n, \mathfrak{T}_5)
            & = \frac{|\mathrm{Aut}(C_5)|}{5!} \lambda_{\alpha}(T_{5,n}^{5}) \\
            & \le \frac{|\mathrm{Aut}(C_5)|}{5!} \left(1+C_{\ref{LEMMA:spectral-balanced-Turan-r-graph}} n^{-2}\right) 5! |T_{5,n}^{5}| n^{-\frac{5}{\alpha}} \\
            & \le \frac{|\mathrm{Aut}(C_5)| |T_{5,n}^{5}|}{n^{5/\alpha}} + |\mathrm{Aut}(C_5)| \left(\frac{n}{5}\right)^{5} n^{-2-\frac{5}{\alpha}}, 
        \end{align*}
        which implies~\eqref{equ:spectral-C5-upper-bound}, since $n$ is large. 
    \end{proof}

    Claim~\ref{CLAIM:C5K3-spectral-balanced} completes the proof of Theorem~\ref{THM:C5-K3-spectral}. 
\end{proof}

\section{An entropy perspective}\label{SEC:Entropy}
%
%
Let $X$ be a discrete random variable taking values in a finite set $\Omega$. For convenience, let 
\begin{align*}
    p_{X}(x) \coloneqq \mathbb{P}(X = x)
    \quad\text{for every}\quad x \in \Omega. 
\end{align*}
The support of $X$ is defined as
\begin{align*}
    \mathrm{Supp}(X)
    \coloneqq \left\{x\in \Omega \colon p_{X}(x) > 0\right\}. 
\end{align*}
For notational convenience, we set $0 \cdot \log_{2}0 = 0$. 
Recall that the well-known \textbf{Shannon entropy}~\cite{Sha48} of $X$ is given by  
\begin{align*}
    \mathbb{H}[X]
    \coloneqq - \sum_{x \in \mathrm{Supp}(X)} p_{X}(x) \cdot \log p_{X}(x)
    = - \sum_{x \in \Omega} p_{X}(x) \cdot \log p_{X}(x). 
\end{align*}
Given random variables $X_1, \ldots, X_q$ defined on a common sample space $\Omega$, their \textbf{mixture}, denoted by $\mu(X_1, \ldots, X_q)$, is the probability distribution defined by
\begin{align*}
    \mathbb{P}\left( \mu(X_1, \ldots, X_q) = v \right)
    = \sum_{i\in [q]} \frac{\mathbb{P}\left(X_i = v\right)}{q}
    \quad\text{for every}\quad v\in \Omega. 
\end{align*}

Let $Q$ be an $r$-graph on the vertex set $[q]$. 
Inspired by the definition in~\cite{CY24}, we say a $q$-tuple of random variables $(X_1, \ldots, X_{q})$ is \textbf{a random embedding of $Q$ in $\mathcal{H}$} if the following conditions are satisfied:
\begin{enumerate}[label=(\roman*)]
    \item\label{prop:random-embedding-a} $(X_1, \ldots, X_{r})$ is always an element in $\mathrm{Inj}(Q,\mathcal{H})$, and 
    \item\label{prop:random-embedding-b} $(X_1, \ldots, X_{r})$ is symmetric with respect to $\mathrm{Aut}(Q)$, that is, $(X_1, \ldots, X_{q})$ is the same as $(X_{\sigma(1)}, \ldots, X_{\sigma(q)})$ for every $\sigma \in \mathrm{Aut}(Q)$. 
\end{enumerate}
We define the \textbf{$(\alpha,Q)$-entropic density} of $\mathcal{H}$ as 
\begin{align*}
    \eta_{\alpha,Q}(\mathcal{H})
    \coloneqq \max_{(X_1, \ldots, X_{q})} 2^{\mathbb{H}(X_1, \ldots, X_{q}) - \frac{q}{\alpha} \mathbb{H}\left( \mu(X_1, \ldots, X_q) \right)}, 
\end{align*}
where the maximum is taken over all random embeddings of $Q$ in $\mathcal{H}$.
In the case where $Q = K_{r}^{r}$ is a single edge, we abbreviate $\eta_{\alpha}(\mathcal{H})$ as $\eta_{\alpha,Q}(\mathcal{H})$.

It was shown in~{\cite[Proposition~5.4]{CY24}} that for every $r$-graph $\mathcal{H}$ and every $\alpha \ge 1$, 
\begin{align*}
    \eta_{\alpha}(\mathcal{H}) = \lambda_{\alpha}(\mathcal{H}).
\end{align*}
We extend their result to arbitrary $Q$. 
\begin{proposition}\label{PROP:Lagrangian-Entropy}
    Let $Q$ and $\mathcal{H}$ be $r$-graphs. 
    For every $\alpha \ge 1$, 
    \begin{align*}
        \eta_{\alpha, Q}(\mathcal{H}) = \lambda_{\alpha, Q}(\mathcal{H}).
    \end{align*}
\end{proposition}
\begin{proof}[Proof of Proposition~\ref{PROP:Lagrangian-Entropy}]
    Fix $\alpha \ge 1$. 
    Let $Q$ be an $r$-graph on the vertex set $[q]$ and $\mathcal{H}$ be an $r$-graph on the vertex set $[n]$. 
    We begin by establishing the following claim.

    \begin{claim}\label{CLAIM:entropy-difference-simplify}
        Let $\mathbf{x} = (x_1, \ldots, x_n) \in \mathbb{R}^{n}$ be a nonnegative vector. 
        Let $(X_1, \ldots, X_n)$ be the random embedding of $Q$ in $\mathcal{H}$ given by the distribution
        \begin{align}\label{equ:random-Q-prob-def}
            \mathbb{P}\left((X_1, \ldots, X_n) = (i_1, \ldots, i_q)\right)
            = \frac{x_{i_1} \cdots x_{i_q}}{\beta}
            \quad\text{for every}\quad (i_1, \ldots, i_q) \in \mathrm{Inj}(Q,\mathcal{H}), 
        \end{align}
        where $\beta \coloneqq \sum_{\phi \in \mathrm{Inj}(Q,\mathcal{H})} \prod_{j \in \phi} x_j$. 
        %
        Also, for every $j \in [n]$, let 
        \begin{align}
            y_j 
            \coloneqq \frac{1}{q} \frac{x_j}{\beta} \sum_{\phi \in \mathrm{Inj}(Q,\mathcal{H}, j)} \prod_{k\in \phi - j} x_k. \label{equ:random-yj-prob-def}
        \end{align}
        Then we have 
        \begin{align*}
            \mathbb{H}[(X_1, \ldots, X_{q})] - \frac{q}{\alpha} \mathbb{H}[\mu(X_1, \ldots, X_q)] 
            = \log_{2} \beta - \frac{q}{\alpha} \sum_{j \in [n]} y_j \log_{2}\left(\frac{x_j^{\alpha}}{y_j}\right). 
        \end{align*}
    \end{claim}
    \begin{proof}[Proof of Claim~\ref{CLAIM:entropy-difference-simplify}]
        By definition~\eqref{equ:random-Q-prob-def}, we have 
        \begin{align*}
            \mathbb{H}[(X_1, \ldots, X_{q})]
            & = \sum_{(i_1, \ldots, i_{q}) \in \mathrm{Inj}(Q,\mathcal{H})} - \frac{x_{i_1} \cdots x_{i_q}}{\beta} \log_{2} \frac{x_{i_1} \cdots x_{i_q}}{\beta} \\
            & = \sum_{(i_1, \ldots, i_{q}) \in \mathrm{Inj}(Q,\mathcal{H})} \frac{x_{i_1} \cdots x_{i_q}}{\beta} \left( \log_{2} \beta - \sum_{j\in [q]} \log_{2}x_{i_j} \right) \\
            & = \log_{2}\beta - \sum_{(i_1, \ldots, i_{q}) \in \mathrm{Inj}(Q,\mathcal{H})} \frac{x_{i_1} \cdots x_{i_q}}{\beta} \sum_{j\in [q]} \log_{2}x_{i_j} \\
            & = \log_{2}\beta - \sum_{j\in [n]} q y_{j} \log_{2} x_j. 
        \end{align*}
        For every $(i,j) \in [q] \times [n]$, let 
        \begin{align*}
             y_{i \to j}
            \coloneqq \frac{x_j}{\beta} \sum_{\phi \in \mathrm{Inj}(Q,\mathcal{H}, i \to j)} \prod_{k\in \phi - j} x_k. 
        \end{align*}
        Note that $P(X_i = j) = y_{i \to j}$ for every $(i,j) \in [q] \times [n]$ and $y_j = \sum_{i \in [q]}y_{i\to j}/q$ for every $j \in [n]$. 
        Therefore, we have 
        \begin{align*}
            \mathbb{H}[\mu(X_{1}, \ldots, X_q)]
            = - \sum_{j \in [n]} y_j \log_2 y_j, 
        \end{align*}
        and hence, 
        \begin{align*}
            & \quad \mathbb{H}\left[ (X_1, \ldots, X_{q}) \right] - \frac{q}{\alpha}  \mathbb{H}[\mu(X_{1}, \ldots, X_q)] \\
            & = \log_{2}\beta - \sum_{j\in [n]} q y_{j} \log_{2} x_j + \frac{q}{\alpha}\sum_{j \in [n]} y_j \log_2 y_j 
            = \log_{2} \beta - \frac{q}{\alpha} \sum_{j \in [n]} y_j \log_{2}\left(\frac{x_j^{\alpha}}{y_j}\right). 
        \end{align*}
        This proves Claim~\ref{CLAIM:entropy-difference-simplify}. 
    \end{proof}

    Now we prove that $\eta_{\alpha, Q}(\mathcal{H}) \le \lambda_{\alpha, Q}(\mathcal{H})$. 
    Fix a nonnegative optimal vector $\mathbf{x} = (x_1, \ldots, x_n) \in \mathrm{OPT}_{\alpha,Q}(\mathcal{H})$, whose existence is guaranteed by Fact~\ref{FACT:nonnegative-optimal-vector}. 
    It follows from the definition of $\mathrm{OPT}_{\alpha,Q}(\mathcal{H})$ that 
    \begin{align*}
        \beta
        \coloneqq \sum_{\phi \in \mathrm{Inj}(Q,\mathcal{H})} \prod_{j \in \phi} x_j
        = \lambda_{\alpha,Q}(\mathcal{H}). 
    \end{align*}
    It follows from the convexity of the function $\log_{2} \colon \mathbb{R}_{>0} \to \mathbb{R}$ that 
    \begin{align*}
        \sum_{j \in [n]} y_j \log_{2}\left(\frac{x_j^{\alpha}}{y_j}\right) 
        \le \log_{2}\left(\sum_{j \in [n]} y_j \cdot \frac{x_j^{\alpha}}{y_j}\right)
        = \log_{2} \left(\sum_{j \in [n]} x_j^{\alpha}\right)
        = \log_{2} 1
        = 0. 
    \end{align*}
    So, by Claim~\ref{CLAIM:entropy-difference-simplify}, we obtain 
    \begin{align*}
        \mathbb{H}\left[ (X_1, \ldots, X_{q}) \right] - \frac{q}{\alpha}  \mathbb{H}[\mu(X_{1}, \ldots, X_q)]
        \ge \log_{2}\beta
        = \log_{2} \lambda_{\alpha,Q}(\mathcal{H}). 
    \end{align*}
    which proves that $\eta_{\alpha, Q}(\mathcal{H}) \ge \lambda_{\alpha, Q}(\mathcal{H})$. 

    Now we consider the other direction. 
    Let $(X_1, \ldots, X_{q})$ be a random embedding of $Q$ in $\mathcal{H}$ that maximizes $\mathbb{H}[(X_1, \ldots, X_{q})] - \frac{q}{\alpha} \mathbb{H}[\mu(X_1, \ldots, X_q)]$. 
    %
    For every $\phi \in \mathrm{Inj}(Q, \mathcal{H})$, let 
    \begin{align*}
        p_{\phi}
        \coloneqq \mathbb{P}\left((X_1, \ldots, X_q) = \phi \right). 
    \end{align*}
    We say that two embeddings $\phi, \psi \in \mathrm{Inj}(Q,\mathcal{H})$ are equivalent, denoted $\phi \sim \psi$, if there exists an automorphism $\theta \in \mathrm{Aut}(Q)$ such that $\phi \circ \theta = \psi$. 
    It follows from Property~\ref{prop:random-embedding-b} that $p_{\phi} = p_{\psi}$ whenever $\phi\sim \psi$. 
    
    Let $\mathcal{N}(Q,\mathcal{H}) \coloneqq \mathrm{Inj}(Q,\mathcal{H})/\sim$ denote the set of equivalence classes in $\mathrm{Inj}(Q,\mathcal{H})$. 
    Let $[X_1, \ldots, X_q]$ denote the probability distribution of $(X_1, \ldots, X_q)$ induced on $\mathcal{N}(Q,\mathcal{H})$, that is, each equivalence class $[\phi] \in \mathcal{N}(Q,\mathcal{H})$, 
    \begin{align}\label{equ:def-q-phi}
        q_{\phi}
        \coloneqq \mathbb{P}\left([X_1, \ldots, X_q] = [\phi] \right)
        = \sum_{\psi \in [\phi]} \mathbb{P}\left((X_1, \ldots, X_q) = \psi\right) 
        = p_{\phi} |\mathrm{Aut}(Q)|. 
    \end{align}
    Then, we have  
    \begin{align}\label{equ:HX1-Xq-expression}
        \mathbb{H}\left[ (X_1, \ldots, X_q) \right]
        & = \mathbb{H}\left[ (X_1, \ldots, X_q) \mid [X_1, \ldots, X_q] \right] + \mathbb{H}\left[ [X_1, \ldots, X_q] \right] \notag \\
        & = \log_{2} |\mathrm{Aut}(Q)| - \sum_{[\phi] \in \mathcal{N}(Q,\mathcal{H})} q_{\phi} \log_{2} q_{\phi}. 
    \end{align}
    For each $j \in [n]$, define  
    \begin{align*}
        \mathcal{N}(Q, \mathcal{H}, j)
        \coloneqq \big\{  [\phi] \in \mathcal{N}(Q, \mathcal{H}) \colon j \in \phi \big\}, 
    \end{align*}
    and let 
    \begin{align}\label{equ:def-yj}
        y_j 
        \coloneqq \frac{1}{q} \sum_{\phi \in \mathrm{Inj}(Q,\mathcal{H}, j)}p_{\phi}
        = \frac{1}{q} \sum_{[\phi] \in \mathcal{N}(Q, \mathcal{H}, j)} q_{\phi}
        \quad\text{and}\quad 
        x_{j}
        \coloneqq  y_j^{\frac{1}{\alpha}}
        = \left(\frac{1}{q} \sum_{[\phi] \in \mathcal{N}(Q, \mathcal{H}, j)} q_{\phi}\right)^{1/\alpha}.  
    \end{align}
    Note that 
    \begin{align*}
        \mathbb{P}\left( \mu(X_1, \ldots, X_q) = j \right)
        = y_j
        \quad\text{for every}\quad j \in [n],  
    \end{align*}
    and thus, 
    \begin{align*}
        \mathbb{H}\left[\mu(X_1, \ldots, X_q)\right]
        = - \sum_{j\in [n]} y_j \log_{2} y_j 
        = - \sum_{j \in [n]} x_j^{\alpha} \log_{2} x_j^{\alpha}. 
    \end{align*}
    Combining it with~\eqref{equ:HX1-Xq-expression}, we conclude that the distribution $\mathbf{q} \coloneqq \left\{ q_{\phi} \colon [\phi] \in \mathcal{N}(Q, \mathcal{H}) \right\}$ maximizes the objective 
    \begin{align*}
        \Phi(\mathbf{q})
        \coloneqq 
        - \sum_{[\phi] \in \mathcal{N}(Q,\mathcal{H})} q_{\phi} \log_{2} q_{\phi} + \frac{q}{\alpha} \sum_{j \in [n]} x_{j}^{\alpha} \log_{2} x_{j}^{\alpha}, 
    \end{align*}
    subject to the constraints 
    \begin{align*}
        q_{\phi} \ge 0 ~\text{for every}~ [\phi] \in \mathcal{N}(Q,\mathcal{H})
        \quad\text{and}\quad 
        \sum_{[\phi] \in \mathcal{N}(Q,\mathcal{H})} q_{\phi} = 1. 
    \end{align*}
    It follows from the Lagrange multiplier method that 
    \begin{align*}
        -1 - \log_{2} q_{\phi} + \frac{1}{\alpha} \sum_{j \in \phi} \left(1+\log_{2} x_j^{\alpha}\right) 
        = -1+\frac{q}{\alpha} - \log_{2} \frac{q_{\phi}}{\prod_{j\in \phi} x_j}
    \end{align*}
    is constant for all $[\phi] \in \mathcal{N}(Q,\mathcal{H})$ with $q_{\phi} > 0$. 
    Equivalently, there exists a constant $\gamma$ such that 
    \begin{align}\label{equ:ratio-is-constant}
        \frac{q_{\phi}}{\prod_{j \in \phi} x_j} = \gamma
    \end{align}
    for all $[\phi] \in \mathcal{N}(Q,\mathcal{H})$ with $q_{\phi} > 0$.

    It follows from~\eqref{equ:def-q-phi} and~\eqref{equ:ratio-is-constant} that 
    \begin{align}\label{equ:p-phi}
        \mathbb{P}\left((X_1, \ldots, X_q) = \phi \right)
        = \frac{q_{\phi}}{|\mathrm{Aut}(Q)|}
        = \frac{\gamma}{|\mathrm{Aut}(Q)|}  \prod_{j \in \phi} x_j
        = \frac{1}{\beta} \prod_{j \in \phi} x_j,  
    \end{align}
    where $\beta \coloneqq |\mathrm{Aut}(Q)|/\gamma$. 
    
    Also, notice that 
    \begin{align*}
        \mathbb{P}\left(\mu(X_1, \ldots, X_q) = j \right) 
        = y_j 
        = \frac{1}{q}\sum_{\phi \in \mathrm{Inj}(Q,\mathcal{H}, j)}p_{\phi}
        = \frac{x_j}{q \beta}  \sum_{\phi \in \mathrm{Inj}(Q,\mathcal{H}, j)} \prod_{k \in \phi - j} x_k.
    \end{align*}
    So it follows from Claim~\ref{CLAIM:entropy-difference-simplify} and definition~\eqref{equ:def-yj} that 
    \begin{align}\label{equ:entropy-Q-lower-bound}
        \mathbb{H}[(X_1, \ldots, X_{q})] - \frac{q}{\alpha} \mathbb{H}[\mu(X_1, \ldots, X_q)] 
        & = \log_{2} \beta - \frac{q}{\alpha} \sum_{j \in [n]} y_j \log_{2}\left(\frac{x_j^{\alpha}}{y_j}\right) \notag \\
        & = \log_{2} \beta - \frac{q}{\alpha} \sum_{j \in [n]} y_j \log_{2} 1
        = \log_{2} \beta.  
    \end{align}
    Note from~\eqref{equ:p-phi} that the partition function $\beta$ satisfies 
    \begin{align*}
        \beta 
        = \sum_{\phi \in \mathrm{Inj}(Q, \mathcal{H})} \prod_{j \in \phi} x_j. 
    \end{align*}
    Also observe that  
    \begin{align*}
        \sum_{i\in [n]} x_i^{\alpha} = \sum_{i \in [n]} y_i = 1
        \quad\text{and}\quad 
        x_i \ge 0 ~\text{for every}~ i \in [n], 
    \end{align*}
    which imply that $(x_1, \ldots, x_n) \in \Delta_{\alpha}^{n-1}$. 
    So, it follows from the definition of $\lambda_{\alpha, Q}(\mathcal{H})$ that $\beta \le \lambda_{\alpha, Q}(\mathcal{H})$. 
    Combining this with~\eqref{equ:entropy-Q-lower-bound}, we conclude that $\eta_{\alpha, Q}(\mathcal{H}) \le \lambda_{\alpha, Q}(\mathcal{H})$. 
    This completes the proof of Proposition~\ref{PROP:Lagrangian-Entropy}. 
\end{proof}

\section{Concluding remarks}\label{SEC:remark}
Just as the theorem of Keevash--Lenz--Mubayi~\cite{KLM14} connects the {A}ndr{\'{a}}sfai--{E}rd{\H o}s--{S}{\'{o}}s-type stability of a Tur\'{a}n problem with its corresponding spectral Tur\'{a}n problem, Theorem~\ref{THM:generalized-spectral} establishes an analogous connection for generalized Tur\'{a}n problems and their spectral counterparts.
Extending the work~\cite{LMR23unified,HLZ24} on Tur\'{a}n problems, the type of stability required by Theorem~\ref{THM:generalized-spectral} was recently studied systematically in~\cite{CL24,CILLP24} for generalized Tur\'{a}n problems, and also investigated in~\cite{CHL24} for certain generalized Tur\'{a}n problems involving odd cycles.
Theorem~\ref{THM:C5-K3-spectral} provides just one application of Theorem~\ref{THM:generalized-spectral}, and we hope this work could motivate further research in this direction.

The entropy approach has been shown to be quite effective in solving certain Tur\'{a}n problems, as shown in works such as~\cite{CY24,L25}. It would be interesting to explore what this approach can produce in the context of the generalized Tur\'{a}n problem.
%
\bibliographystyle{alpha}
\bibliography{SpectralTuran}
\appendix
\section{Basic inequalities and lemmas}
In this appendix, we provide proofs of several inequalities used in the main arguments, as well as some basic lemmas analogous to those in~\cite{KLM14}. Although these lemmas are not directly used in the proofs, we include them for future research.

\begin{fact}\label{FACT:inequality-c}
    Suppose that $x > 1$ and $\beta > 0$. Then 
    \begin{align*}
        \left(\frac{x}{x-1}\right)^{\beta} 
        \ge 1 + \frac{\beta}{x}. 
    \end{align*}
\end{fact}
\begin{proof}[Proof of Fact~\ref{FACT:inequality-c}]
    Fix $\beta > 0$ and let 
    \begin{align*}
        f(x) \coloneqq \left(\frac{x}{x-1}\right)^{\beta} - \left( 1 + \frac{\beta}{x} \right). 
    \end{align*}
    Straightforward calculations show that
    \begin{align*}
        \frac{\mathrm{d} f(x)}{\mathrm{d} x}
        = -\beta  \left(x \left(\frac{x}{x-1}\right)^{\beta }+x-1\right){(x-1)^{-1} x^{-2}}.  
    \end{align*}
    It follows that $f(x)$ is decreasing on the interval $(1, \infty)$. 
    Therefore, $f(x) \ge \lim_{x \to \infty} f(x) = 0$, which implies the desired inequality.
\end{proof}

\begin{fact}\label{FACT:inequality-a}
    Suppose that $x\in [0,1]$ and $\beta \ge 0$. Then 
    $\left(1-x\right)^{-\beta} \ge 1 + \beta x$. 
\end{fact}
\begin{proof}[Proof of Fact~\ref{FACT:inequality-a}]
    Fix $\beta > 0$ (note that the case $\beta = 0$ is trivial) and let 
    \begin{align*}
        f(x)
        \coloneqq (1+\beta x)(1-x)^{\beta} - 1. 
    \end{align*}
    Straightforward calculations show that
    \begin{align*}
        \frac{\mathrm{d} f(x)}{\mathrm{d} x}
        = - \beta  (\beta +1) x (1-x)^{\beta -1}. 
    \end{align*}
    It follows that $f(x)$ is decreasing on the interval $[0,1]$. 
    Therefore, $f(x) \le f(0) = 0$, which implies the desired inequality.
\end{proof}

\begin{fact}\label{FACT:inequality-b}
    Suppose that $x\in [0,1/2]$. Then 
    $1-x \ge e^{-x-x^2}$. 
\end{fact}
\begin{proof}[Proof of Fact~\ref{FACT:inequality-b}]
    Let $f(x) \coloneqq (1-x)e^{x+x^2}$. Straightforward calculations show that 
    \begin{align*}
        \frac{\mathrm{d} f(x)}{\mathrm{d} x} = x(1-2x)e^{x^2+x}. 
    \end{align*}
    It follows that $f(x)$ is increasing on the interval $[0, 1/2]$. 
    Therefore, $f(x) \ge f(0) = 1$, which implies the desired inequality.
\end{proof}

Let $\mathcal{H}$ be an $r$-graph on the vertex set $[n]$. Two vertices $i$ and $j$ are said to be \textbf{equivalent} in $\mathcal{H}$ if the transposition $(ij)$ is an automorphism of $\mathcal{H}$, that is, if the map $\phi$ defined by 
\begin{align*}
    \phi(k)
    = 
    \begin{cases}
        j, &\quad\text{if}\quad k = i, \\
        i, &\quad\text{if}\quad k = j, \\
        k, &\quad\text{if}\quad k \in [n]\setminus\{i,j\}, \\
    \end{cases}
\end{align*}
is an automorphism of $\mathcal{H}$.

\begin{lemma}\label{LEMMA:KLM-transposition-extend}
    Let $\alpha \ge 1$ be a real number and $Q$ be an $r$-graph. 
    Let $\mathcal{H}$ be an $r$-graph on the vertex set $[n]$. 
    Suppose that vertices $i$ and $j$ are equivalent in $\mathcal{H}$. 
    Let $\mathbf{x} = (x_1, \ldots, x_n) \in \Delta^{n-1}_{\alpha}$.
    Define $\mathbf{y} = (y_1, \ldots, y_n) \in \Delta^{n-1}_{\alpha}$ by
    \begin{align}\label{equ:symm-vector}
        y_k 
        =
        \begin{cases}
            \left(\frac{x_i^{\alpha} + x_j^{\alpha}}{2}\right)^{1/\alpha}, & \quad\text{if}\quad k \in \{i,j\}, \\
            x_k, & \quad\text{if}\quad k \in [n] \setminus \{i,j\}.
        \end{cases}
    \end{align}
    Then 
    \begin{align*}
        P_{Q,\mathcal{H}}(\mathbf{y})
        \ge P_{Q,\mathcal{H}}(\mathbf{x}). 
    \end{align*}
\end{lemma}
\begin{proof}[Proof of Lemma~\ref{LEMMA:KLM-transposition-extend}]
    Since $\alpha \ge 1$, the Power Mean Inequality implies that 
    \begin{align*}
        y_i = y_j \ge \sqrt{x_i x_j}
        \quad\text{and}\quad 
        y_i + y_j \ge x_i + x_j. 
    \end{align*}
    Let 
    \begin{align*}
        C_{i}
        & \coloneqq \left\{\phi \in \mathrm{Inj}(Q,\mathcal{H}) \colon \text{$i \in \phi(V(Q))$ but $j \not\in \phi(V(Q))$} \right\}, \\[0.3em]
        C_{j}
        & \coloneqq \left\{\phi \in \mathrm{Inj}(Q,\mathcal{H}) \colon \text{$j \in \phi(V(Q))$ but $i \not\in \phi(V(Q))$} \right\}, \\[0.3em]
        C_{ij}
        & \coloneqq \left\{\phi \in \mathrm{Inj}(Q,\mathcal{H}) \colon \text{$i \in \phi(V(Q))$ and $j \in \phi(V(Q))$} \right\}.
    \end{align*}
    By definition, we have 
    \begin{align*}
        P_{Q,\mathcal{H}}(\mathbf{y}) - P_{Q,\mathcal{H}}(\mathbf{x})
        & = \sum_{\phi \in C_i} \left(y_i - x_i\right) \prod_{k \in \phi - i} x_k \\
        & \quad  +  \sum_{\phi \in C_j} \left(y_j - x_j\right) \prod_{k \in \phi - j} x_k 
          +  \sum_{\phi \in C_{ij}} \left(y_iy_j - x_ix_j\right) \prod_{k \in \phi - \{i,j\}} x_k. 
    \end{align*}
    Since $i$ and $j$ are equivalent in $\mathcal{H}$, we have $C_i = C_j$. 
    It follows that 
    \begin{align*}
        P_{Q,\mathcal{H}}(\mathbf{y}) - P_{Q,\mathcal{H}}(\mathbf{x})
        & = \sum_{\phi \in C_i} \left(y_i + y_j - x_i - x_j\right) \prod_{k \in \phi - i} x_k \\
        & \quad  +  \sum_{\phi \in C_{ij}} \left(y_iy_j - x_ix_j\right) \prod_{k \in \phi - \{i,j\}} x_k
        \ge 0,
    \end{align*}
    which proves Lemma~\ref{LEMMA:KLM-transposition-extend}. 
\end{proof}

\begin{lemma}\label{LEMMA:KLM-equivalence-classes-extend}
    Let $\alpha \ge 1$ be a real number. Let $\mathcal{H}$ be an $r$-graph on $[n]$, and let $P_1 \cup \cdots \cup P_k = [n]$ be a partition such that each $P_i$ is an equivalence class in $\mathcal{H}$. Then there exists a vector $\mathbf{x} = (x_1, \ldots, x_n) \in \mathrm{OPT}_{\alpha, Q}(\mathcal{H})$ such that $x_i = x_j$ whenever $i$ and $j$ lie in the same equivalence class.
\end{lemma}
\begin{proof}[Proof of Lemma~\ref{LEMMA:KLM-equivalence-classes-extend}]
    Choose $\mathbf{x} = (x_1, \ldots, x_n) \in \mathrm{OPT}_{\alpha, Q}(\mathcal{H})$ such that 
    \begin{align*}
        \Phi(\mathbf{x})
        \coloneqq \sum_{j\in [k]} \sum_{i \in P_j} \left| x_i - \overline{x}_{P_j}\right|
    \end{align*}
    is minimized, where 
    \begin{align*}
        \overline{x}_{P_j}
        \coloneqq \left(|P_j|^{-1} \sum_{i \in P_j} x_i^{\alpha}\right)^{1/\alpha}. 
    \end{align*}
    Note that it suffices to show that $\Phi(\mathbf{x}) = 0$. Suppose to the contrary that this fails. 
    Then there exists some $P_{\ell}$ such that $x_i \neq x_j$ for some $\{i,j\} \subseteq P_{\ell}$. 
    Let $\mathbf{y} = (y_1, \ldots, y_n)$ be defined as in~\eqref{equ:symm-vector}. 
    It follows from Lemma~\ref{LEMMA:KLM-transposition-extend} that $\mathbf{y} \in \mathrm{OPT}_{\alpha, Q}(\mathcal{H})$. 
    However, $\Phi(\mathbf{y}) < \Phi(\mathbf{x})$, which contradicts the choice of $\mathbf{x}$.
\end{proof}

\end{document}